\documentclass[11pt]{amsart}
\usepackage{amssymb}
\usepackage{amsmath}

\newcommand{\ck}{{\bar{\bk}}}

\newcommand{\f}{\textbf}

\newtheorem{numbered}{}[section]
\newtheorem{Theorem}[numbered]{Theorem}

\newtheorem{nothing*}[numbered]{}
\newtheorem{remark}[numbered]{Remark}

\newtheorem{Definition}[numbered]{Definition}

\newtheorem{Lemma}[numbered]{Lemma}
\newtheorem{Proposition}[numbered]{Proposition}
\newtheorem{Corollary}[numbered]{Corollary}
\newtheorem{example}[numbered]{Example}

\def\bc{{\mathbb C}}

\def\bm{{\mathbb M}}
\def\bn{{\mathbb N}}
\def\bp{{\mathbb P}}

\def\br{{\mathbb R}}

\def\l{\lambda}

\def\i{\varepsilon}
\def\t{\tau}
\def\f{\varphi}

\def\N{\mathbb{N}}

\def\a{\alpha}

\def\b{\beta}
\def\m{\mu}
\def\n{\nu}
\def\d{\delta}
\def\g{\gamma}

\def\ca{\mathcal{ A}}

\def\ck{\mathcal{K}}

\def\cq{\mathcal{ Q}}

\def\ga{{\frak A}}

\def\id{{\bf 1}\!\!{\rm I}}

\begin{document}

\begin{center}
{\Large {\bf Generalized Dobrushin Ergodicity Coefficient and Uniform Ergodicities of Markov Operators}}\\[1cm]

{\sc Farrukh Mukhamedov} \\[2mm]

 Department of Mathematical Sciences, College of Science, \\
 United Arab Emirates University 15551, Al-Ain,\\ United
Arab Emirates.\\

e-mail: {\tt far75m@yandex.ru; farrukh.m@uaeu.ac.ae}\\[1cm]

{\sc  Ahmed Al-Rawashdeh} \\[2mm]

Department of Mathematical Sciences, College of Science, \\
United Arab Emirates University 15551, Al-Ain, \\ United
Arab Emirates.\\
 e-mail: {{\tt aalrawashdeh@uaeu.ac.ae}}\\[1cm]
\end{center}

\small
\begin{center}
{\bf Abstract}\\
\end{center}

In this paper the stability and the perturbation bounds of Markov operators acting on abstract state spaces are investigated.  Here, an abstract state space is an ordered Banach space where the norm has an additivity property on the cone of positive elements. We basically study uniform ergodic properties of Markov operators by means of so-called a generalized Dobrushin's ergodicity coefficient. This allows us to get several convergence results with rates. Some results on quasi-compactness of Markov operators are proved in terms of the ergodicity coefficient. Furthermore, a characterization of uniformly $P$-ergodic Markov operators is given which enable us to construct plenty examples of such types of operators. The uniform mean ergodicity of Markov operators is
established in terms of the Dobrushin ergodicity coefficient. The obtained results are even new in the classical and quantum settings.\\[2mm]

\normalsize

\section{Introduction}

It is known that Doeblin and Dobrushin \cite{D,IS} characterized the contraction rate of
Markov operators which act
on a space of measures equipped with the total variation norm as follows: Let us consider a finite Markov chain with a transition (row stochastic) matrix
$\bp=(p_{ij})\in \br^{n\times n}$. It defines a Markov operator  $P:\br^n\to \br^n$ such that $P\mathbf{x}=\mathbf{x}\bp$, where the elements of  $\br^n$ are row vectors.
The set of probability measures can be identified with the standard simplex $\ck=\{(x_i)\in\br^n: \ x_i\geq 0, \ \sum_{i=1}^n x_i=1\}$. The total variation norm is
nothing but one half of the $\ell_1$ norm $\|\cdot\|_1$ on $\br^n$. One can introduce the following coefficient
$$
\d(P)=\sup_{\m,\n\in\ck, \m\neq \n}\frac{\|P\m-P\n\|_1}{\|\m-\n\|_1}.
$$

This coefficient is characterized by Doeblin and Dobrushin \cite{D} as follows:
\begin{eqnarray}\label{DD1}
\d(P)&=&\frac{1}{2}\max_{i<j}\sum_{k=1}^n|p_{ik}-p_{jk}|\\
&=&1-\min_{i<j}\sum_{k=1}^n\min\{p_{ik},p_{jk}\}.
\end{eqnarray}

It is known that if $\d(P)<1$ (this condition is often called \textit{Dobrushin condition}) then $P^n$ converges to its invariant distribution with exponential rate \cite{D,Se2}.  Moreover, this condition also gives
the spectral gap of the operator $P$ (see \cite{Se2}).
The Dobrushin condition played a major role as a source of inspiration for many mathematicians to do interesting work on the theory of Markov processes (see for example \cite{IS,MT,Se2}).

Let us consider the following example:  Let $T: \br^3 \to \br^3$ be the Markov operator which is given by the matrix
\[\begin{pmatrix}
1 &  0  & 0 \\
0 & 1  & 0 \\
 0 &\frac{1}{2} & \frac{1}{2}
\end{pmatrix}
.\]
It is clear that $T^n$ converges to $P$, where
\[ P = \begin{pmatrix}
1 &  0  & 0 \\
0 & 1  & 0 \\
 0 & 1& 0
\end{pmatrix}
.\]
One can calculate that $\d(T)=1$. From this, we infer that $T^n$ converges, but $\d(T)=1$.  Hence, the investigation of the sequence $\{T^n\}$ in terms of $\d(T)$ is not effective.
Hartfiel et al. \cite{H,HR} introduced a generalized coefficient which covers the mentioned type of convergence in the finite-dimensional setting.  To the best knowledge of the authors, such coefficient is not
studied even in the classical $L^1$-spaces. Therefore, the main aim of this paper is to define an  analogue of the coefficient mentioned above in a more general setting, i.e. for ordered Banach spaces, such that it will cover all known classical spaces as particular cases.
Moreover, we are going to investigate uniform asymptotic stabilities of Markov operators on ordered Banach spaces. We notice that the consideration of these types of Banach spaces is convenient and important for the study of several properties of physical and
probabilistic processes in an abstract framework which covers the classical and quantum cases (see \cite{Alf,E}). In this
setting, certain limiting behaviors of Markov operators were
investigated in \cite{ALM, B,EW2,GQ,SZ}.

Our purpose is to investigate stability and perturbation bounds of
Markov operators acting on abstract state spaces. More precisely, an abstract state space is an ordered Banach space where the norm has an additivity property on the cone of positive elements. Examples of these spaces include all classical $L^1$-spaces and the space of density operators acting on some Hilbert spaces \cite{Alf,Jam}. Moreover, any Banach space can be embedded into some abstract spaces (see Example \ref{E1} (c)).
There are a few results in the literature on uniform convergence of iterates of bounded linear operators on Banach spaces (see, e.g. \cite{E,Jach,L,LSS,MZ,R,TZ}).
In the present paper, we study the asymptotic stability (in the sense of uniform topology) of Markov operators based on the so-called generalized Dobrushin's ergodicity coefficient. This allows us
to get several convergence results with rates.
  We notice that the Dobrushin coefficient (which extends $\d(P)$ to  abstract state spaces) has been introduced and studied in \cite{GQ,M0,M01}, for Markov
operators  acting on abstract state spaces.

The paper is organized as follows. In Section 2, we provide preliminary definitions and results on properties of abstract state spaces. In Section 3, we define a generalized Dobrushin ergodicity
coefficient $\d_P(T)$ of Markov operators with respect to a projection $P$ and study its properties. Some results on quasi-compactness of Markov operators are proved in terms of this coefficient.
At the end of that section, we give some connection of $\d_P(T)$ to the spectral gap of $T$. Furthermore, in Section 4,  the uniform $P$-ergodicity of Markov operators is studied in terms of the generalized Dobrushin ergodicity
coefficient. This allows us to establish certain category results for the set of uniformly $P$-ergodic Markov operators.  An application of the main result of this section is to get results on uniform ergodicities of linear bounded operators on Banach spaces.
In Section 5, we give a characterization of uniformly $P$-ergodic Markov operators which enables us to explicitly construct  such operators. Finally, in Section 6,
we establish perturbation bounds
for  the uniform $P$-ergodic Markov operators. It is noticed that perturbation bounds have important applications in the theory of probability and quantum information (see, \cite{EM2018,Mit,Mit1,SW}). Moreover, the
results are even new in the classical and quantum
settings.

\section{Preliminaries}

In this section, we recall some necessary definitions and results
about abstract state spaces.

 Let $X$ be an ordered vector space
with a cone $X_+=\{x\in X: \ x\geq 0\}$. A subset $\ck$ is called
a {\it base} for $X$, if  $\ck=\{x\in X_+:\ f(x)=1\}$ for
some strictly positive (i.e. $f(x)>0$ for $x>0$) linear functional
$f$ on $X$. An ordered vector space $X$ with generating cone $X_+$
(i.e. $X=X_+-X_+$) and a fixed base $\ck$, defined by a functional
$f$, is called {\it an ordered vector space with a base}
\cite{Alf}.  Let
$U$ be the convex hull of the set $\ck\cup(-\ck)$, and let
$$
\|x\|_{\ck}=\inf\{\l\in\br_+:\ x\in\l U\}.
$$
Then one can see that $\|\cdot\|_{\ck}$ is a seminorm on $X$.
Moreover, one has $\ck=\{x\in X_+: \ \|x\|_{\ck}=1\}$,
$f(x)=\|x\|_{\ck}$ for $x\in X_+$.
Assume that the seminorm becomes a norm, and $X$ is complete space w.r.t. this norm and $X_+$ is closed subset, then
$(X,X_+,\ck,f)$ is called \textit{abstract state
space}. In this case, $\ck$ is a closed face of the unit ball of $X$,  and $U$ contains the open
unit ball of $X$.
If the set $U$ is \textit{radially compact} \cite{Alf}, i.e.
$\ell\cap U$ is a closed and bounded segment for every line $\ell$
through the origin of $X$,  then  $\|\cdot\|_{\ck}$ is a
norm.  The radial compactness is equivalent to the coincidence of $U$ with the closed unit ball of $X$.  In this case,
$X$  is called a \textit{strong abstract state
space}.  In the sequel, for the sake of simplicity, instead of
$\|\cdot\|_{\ck}$, the standard notation $\|\cdot\|$ is used.
To better understand the difference between a strong abstract state space and a more general class of base norm spaces, the reader is referred to
\cite{Yo}.

A positive cone $X_+$  of an ordered Banach space $X$  is said to be $\l$-generating if, given
$x\in X$, we can find $y,z\in X_+$ such that $x = y- z$ and $\|y\| + \|z\|\leq\l \|x\|$.
The norm on $X$
is  called \textit{regular} (respectively, \textit{strongly regular}) if, given $x$ in the open (respectively,
closed) unit ball of $X$,  $y$ can be found in the closed unit ball with $y\geq  x$ and $y\geq -x$.
The
norm is said to be additive on $X_+$ if $\|x + y\| = \|x\| + \|y\|$ for all $ x, y\in X_+$.
If $X_+$  is 1-generating, then  $X$ can be shown to be strongly regular. Similarly, if $X_+$ is
$\l$-generating for all $\l > 1$, then $X$ is regular \cite{Yo}.
The following results are well-known.

\begin{Theorem}\cite[p.90]{WN} Let  $X$ be an ordered Banach space with closed positive cone $X_+$.
Then te following statements are equivalent:
\begin{itemize}
\item[(i)] $X$ is an abstract state space;
\item[(ii)] $X$ is regular, and the norm is additive on $X_+$;
\item[(iii)] $X_+$ is $\l$-generating for all $\l > 1$, and the norm is additive on $X_+$.
\end{itemize}
\end{Theorem}

\begin{Theorem}\cite{Yo}\label{Yo} Let  $X$ be an ordered Banach space with closed positive cone $X_+$.
Then the following statements are equivalent:
\begin{itemize}
\item[(i)] $X$ is a strong abstract state space;
\item[(ii)] $X$ is strongly regular, and the norm is additive on $X_+$;
\item[(iii)] $X_+$ is 1-generating and the norm is additive on $X_+$.
\end{itemize}
\end{Theorem}

In this paper, we consider a general abstract state space for which the convex hull of the base $\ck$ and $-\ck$ is not
assumed to be radially compact (in our previous papers \cite{EM2017,M0,M01} this condition was essential). This consideration has an important advantage: whenever $X$ is an ordered Banach space with
a generating cone $X_+$ whose norm is additive on $X_+$, then $X$ admits an equivalent norm that coincides
with the original norm on $X_+$ and renders $X$ that base norm space. Hence, to apply the results of the paper
one would then only have to check that if the norm is additive on $X_+$.

\begin{example}\label{E1} Let us provide some examples of
abstract state spaces.
\begin{itemize}

 \item[(a)]  Let $M$ be a von Neumann algebra. Let $M_{h,*}$ be the
Hermitian part of the predual space $M_*$ of $M$. As a base $\ck$ we
define the set of normal states of $M$. Then
$(M_{h,*},M_{*,+},\ck,\id)$ is a strong abstract state spaces, where $M_{*,+}$ is the set of
all positive functionals taken from $M_*$, and $\id$ is the unit in
$M$. In particular, if $M=L^\infty(E,\m)$, then $M_{*}=L^1(E,\m)$ is an abstract state space.

\item[(b)]
Let $A$ be a real ordered linear space and, as before, let $ A_+$ denote
the set of positive elements of $A$. An element $e\in A_+$ is called
\textit{order unit} if for every $a\in A$ there exists a number
$\l\in\br_+$ such that $-\l e\leq a\leq\l e$. If the order is
Archimedean, then the mapping $a\to\|a\|_e=\inf\{\l> 0\ : \-\l e\leq
a\leq\l e\}$  is a norm. If $A$ is a Banach space with respect to
this norm, the pair $(A, e)$ is called \textit{an order-unit
space with the order unit $e$}.
An element $\rho\in A^*$ is called \textit{positive} if
$\rho(x)\geq 0$ for all $a\in A_+$. By $A^*_+$ we denote the set of
all positive functionals. A positive linear functional is called a
\textit{state} if $\rho(e)=1$. The set of all states is denoted by
$S(A)$. Then it is well-known that $(A^*,A^*_+,S(A),e)$
is a strong abstract state space \cite{Alf}. In particular, if $\ga_{sa}$ is the self-adjoint part of an unital $C^*$-algebra, $\ga_{sa}$ becomes order-unit spaces, hence  $(\ga_{sa}^*,\ga_{sa,+}^*,S(\ga_{sa}),\id)$
is a strong abstract state space.

\item[(c)] Let $X$ be a Banach space over $\br$. Consider a new Banach space
$\mathcal{X}=\br\oplus X$ with a norm $\|(\a,x)\|=\max\{|\a|,\|x\|\}$.
Define a cone $\mathcal{X}_+=\{(\a,x)\ : \ \|x\|\leq \a, \
\a\in\br_+\}$ and a positive functional $f(\a,x)=\a$. Then one can
define a base $\ck=\{(\a,x)\in\mathcal{X}:\ f(\a,x)=1\}$. Clearly, we
have $\ck=\{(1,x):\ \|x\|\leq 1\}$. Then $(\mathcal{X},\mathcal{X}_+,\ck,f)$ is an abstract state space \cite{Jam}. Moreover, $X$
can be isometrically embedded into $\mathcal{X}$.  Using this construction one can study several interesting examples of abstract state spaces.

\item[(d)]  Let $A$ be the disc algebra, i.e. the sup-normed space of complex-valued
functions which are continuous on the closed unit disc, and analytic on the open unit disc.
Let $X =\{f\in A :\ f(1)\in\br \}$. Then $X$ is a real Banach space with the following positive cone
$X_+=\{f\in X: f(1)=\|f\|\}=\{f\in X: \ f(1)\geq\|f\|\}$. The space $X$ is an abstract state space, but not strong one (see  \cite{Yo} for details).
\end{itemize}

\end{example}

Let $(X,X_+,\ck,f)$ be an abstract state space. A linear operator $T:X\to X$ is
called \textit{positive}, if $Tx\geq 0$ whenever $x\geq 0$. A positive linear
operator $T:X\to X$ is said to be {\it Markov}, if $T(\ck)\subset\ck$.
It is clear that $\|T\|=1$, and its adjoint operator $T^*: X^*\to
X^*$ acts an ordered Banach space $X^*$ with unit $f$, and moreover,
$T^*f=f$.
 Now for each
$y\in X$ we define a linear operator $T_y: X\to X$ by
$T_y(x)=f(x)y$.

From the definition of Markov operator, one can prove the following auxiliary fact.

\begin{Lemma}\label{Lem0}
Let $(X,X_+,\ck,f)$ be an abstract state space and let $T$ be a Markov operator on $X$. Then for any $x\in X$, we have $f(Tx)=f(x)$.
\end{Lemma}

\begin{example}\label{E2} Let us consider several examples of Markov operators.

\begin{itemize}

\item[1.] Let $X=L^1(E,\m)$ be the classical $L^1$-space. Then any transition probability $P(x,A)$ defines a Markov operator $T$ on $X$, whose dual $T^*$ acts on $L^\infty(E,\m)$ as follows
$$
(T^*f)(x)=\int f(y)P(x,dy), \ \ f\in L^\infty.
$$

\item[2.] Let $M$ be a von Neumann algebra, and consider $(M_{h,*},M_{*,+},\ck,\id)$ as in (a) Example \ref{E1}.
Let $\Phi: M\to M$ be a positive, unital ($\Phi(\id)=\id$) linear mapping. Then the operator given by $(Tf)(x)=f(\Phi(x))$, where $f\in M_{h,*}, x\in M$, is a Markov operator.

\item[3.] Let $X=C[0,1]$ be the space of real-valued continuous
functions on $[0,1]$. Denote
$$
X_+=\big\{x\in X: \ \max_{0\leq t\leq 1}|x(t)-x(1)|\leq 2
x(1)\big\}.
$$
Then $X_+$ is a generating cone for $X$, and $f(x)=x(1)$ is a
strictly positive linear functional. Then $\ck=\{x\in X_+: \
f(x)=1\}$ is a base corresponding to $f$. One can check that the
base norm $\|x\|$ is equivalent to the usual one
$\|x\|_{\infty}=\max\limits_{0\leq t\leq 1}|x(t)|$. Due to
closedness of $X_+$ we conclude that $(X,X_+,\ck,f)$ is an abstract state space.
Let us define a mapping $T$ on $X$ as follows:
$$(Tx)(t)=tx(t).
$$
It is clear that $T$ is a
Markov operator on $X$.

\item[4.] Let $X$ be a Banach space over $\br$. Consider the abstract state space $(\mathcal{X},\mathcal{X}_+,\tilde \ck,f)$ constructed in (c) Example \ref{E1}.  Let $T:X\to X$ be a linear bounded operator with $\|T\|\leq 1$.
Then the operator $\mathcal{T}: \mathcal{X}\to\mathcal{X}$ defined by
$\mathcal{T}(\a,x)=(\a,Tx)$ is a Markov operator.

\item[5.]  Let $A$ be the disc algebra, and let $X$ be the abstract state space as in (d) Example \ref{E1}.
A mapping $T$ given by $Tf(z)=zf(z)$ is clearly a Markov operator on $X$.
 \end{itemize}
\end{example}

\begin{Definition}\cite{M0}
Let $(X,X_+,\ck,f)$ be an abstract state space, and let $T:X\to X$ be a Markov operator.
Then the \textit{Dobrushin's ergodicity
coefficient} of $T$ is given by
\begin{equation}
\label{db} \d(T)=\sup_{x\in N,\ x\neq 0}\frac{\|Tx\|}{\|x\|},
\end{equation}
where
\begin{equation}
\label{NN} N=\{x\in X: \ f(x)=0\}.
\end{equation}
\end{Definition}

\begin{remark} We note that if $X=L^1(E,\m)$, the notion of
the Dobrushin ergodicity coefficient was studied in
\cite{C} and \cite{D}. In a non-commutative setting, i.e. when $X^*$
is a von Neumann algebra, such a notion was introduced in
\cite{M}. We should stress that this coefficient has been
independently defined in \cite{GQ}.
\end{remark}

\section{Generalized Dobrushin Ergodicity Coefficient}

In this section, we introduce a generalized notion of the Dobrushin's ergodicity coefficient \eqref{db}, and investigate its properties.

\begin{Definition} Let $(X,X_+,\ck,f)$ be an abstract state space and let $T:X\to X$ be a linear bounded
operator. Consider a non-trivial projection operator $P: X\to X$ (i.e. $P^2=P$). Then we define
\begin{equation} \label{Dbp} \d_P(T)=\sup_{x\in N_P,\ x\neq
0}\frac{\|Tx\|}{\|x\|},
\end{equation}
where
\begin{equation}\label{Np} N_P=\{x\in X: \
Px=0\}.
\end{equation}
If $P=I$, we put $ \d_P(T)=1$. The quantity $\d_P(T)$ is called the \textit{generalized Dobrushin ergodicity
coefficient of $T$ with respect to $P$}.
\end{Definition}

 We notice that if $X=\br^n$, then there are some formulas to calculate this coefficient (see \cite{H,HR}).

In the following remarks, let us have a brief comparison between the coefficients $\d_P(T)$ and $\d(T)$.

\begin{remark}
Let $y_0\in \ck$ and consider the projection $Px=f(x)y_0$. Then one can see that $N_P$ coincides with
\[N= \{ x\in X;\ f(x)=0\},\]
and in this case $\d_P(T)=\d(T)$. Hence, $\d_P(T)$ indeed is a generalization of $\d(T)$.
\end{remark}

\begin{remark}
Let $P$ be a Markov projection on $X$. Then, for any Markov operator $T:X\to X$
\[\d_P(T) \leq \d(T).\]
Indeed, it is enough to show that $N_P\subseteq N$. Let $x\in N_P$, so $Px=0$. Due to Lemma \ref{Lem0}, we have
\[N=\{x\in X;\ f(Px)=0\},\]
which yields $x\in N$, so $N_P\subseteq N$.

\end{remark}

In what follows, we examine main properties of $\d_P(T)$.

\begin{Proposition}
Let $T:X\to X$ be a linear bounded
operator. If $P$ and $Q$ are two projections on $X$ such that $Q\leq P$ (i.e. $QP=PQ=Q$), then $\d_P(T)\leq \d_Q(T)$.
\end{Proposition}
\begin{proof}
Assume that $Q\leq P$. Then for every $x\in N_P$ we get $Qx=QPx=0$, therefore $N_P\subseteq N_Q$. Hence, we get the desired inequality.
\end{proof}

\begin{Corollary}
If $P$ and $Q$ are orthogonal projections on $X$, then $\d_{P+Q}(T)\leq \d_P(T)$.
\end{Corollary}
\begin{proof}
As $P$ and $Q$ are orthogonal projections, $P+Q$ is a projection which dominates $P$, hence the corollary follows directly from the previous proposition.
\end{proof}

Before establishing our main result of this section, we need the following auxiliary fact.

\begin{Lemma}\label{Lem2} Let $(X,X_+,\ck,f)$ be an abstract state space and let $P$ be a Markov projection. Then for every $x\in N_P$ there
exist $u,v\in \ck$ with $u-v\in N_P$ such that
$$
x=\a(x)(u-v),
$$
where $\a(x)\in \br_+$ and $\a(x)\leq \frac{\l}{2}\|x\|$.
\end{Lemma}

\begin{proof} Given any $x\in N_P$, we have $Px=0$.  As $X_+$ is $\l$-generating of $X$, there exist $x_+,x_-\in X_+$ such that $x=x_+-x_-$ with $\|x_+\|+\|x_-\|\leq \l \|x\|$. Clearly
$Px_+=Px_-$. As $P$ a Markov projection
\[\|Px_+\|=f(Px_+)=f(x_+)=\|x_+\|,\]
which yields $\|x_+\|=\|x_-\|$. Therefore,
\begin{eqnarray*}
x &=& \frac{x_+}{\|x_+\|}\|x_+\|- \frac{x_-}{\|x_-\|}\|x_+\|\\
&=& \|x_+\|\left (\frac{x_+}{\|x_+\|}- \frac{x_-}{\|x_-\|}  \right ).
\end{eqnarray*}
letting $u=\frac{x_+}{\|x_+\|}$ and $v=\frac{x_-}{\|x_-\|}$, so $u,v\in \ck$.  Moreover,  $Pu=Pv$,  then $u-v\in N_P$, and letting $\a(x):=\|x_+\|\leq \frac{\l}{2}\|x\|$, hence the lemma is proved.\\
\end{proof}

Let us denote by $\Sigma(X)$ the set of all Markov operators defined
on $X$, and by $\Sigma_P(X)$ we denote the set of all Markov operators
$T$ on $X$ with $PT=TP$.\\

Now, we prove the following essential result about main properties of $\d_P$.

\begin{Theorem}\label{Dbp-prp} Let $(X,X_+,\ck,f)$ be an abstract state space, $P$ be a projection on $X$ and let $T,S\in \Sigma(X)$. Then:
\begin{enumerate}
\item[(i)] $0\leq \d_P(T)\leq 1$;
\item[(ii)]  $|\d_P(T)-\d_P(S)|\leq\d_P(T-S)\leq \|T-S\|$;
\item[(iii)] if $P\in \Sigma(X)$,  one has
\begin{equation}\label{dpuv} \d_P(T)\leq \frac{\l}{2}\sup \{\|Tu-Tv\|;\ u,v\in \ck \ \text{with}\ u-v\in N_P \}.
\end{equation}
\item[(iv)]  if $H: X\to X$ is a bounded linear
operator such that $HP=PH$, then $$\d_P(TH)\leq\d_P(T)\|H\|;$$
\item[(v)]   if $H: X\to X$ is a bounded linear
operator such that $PH=0$, then $$\|TH\|\leq\d_P(T)\|H\|;$$
\item[(vi)]  if $S\in \Sigma_P(X)$, then $$\d_P(TS)\leq\d_P(T)\d_P(S).$$
\end{enumerate}
\end{Theorem}

\begin{proof}
(i) As $T$ is a Markov operator and by the definition of $\d_P$ one gets $0\leq \d_P(T)\leq \|T\|= 1$. (ii) The second inequality is immediately obtained from \eqref{Dbp}. To establish the first  one, take any $\epsilon >0$.
Then there exists an $x_\epsilon \in N_P$, with $\|x_\epsilon \|=1$ such that $\d_P(T)\leq \|Tx_\epsilon \|+\epsilon$.
Hence,
\begin{eqnarray*}
\d_P(T)-\d_P(S)&\leq &\|Tx_\i\|+\i-\sup_{x\in N_P,\|x\|=1}\|Sx\|\\
&\leq & \|Tx_\i\|-\|S x_\i\|+\i\\
&\leq&\|(T-S)x_\i\|+\i\\
&\leq&\sup_{x\in N_P: \ \|x\|=1}\|(T-S)x\|+\i\\
&=&\d_P(T-S)+\i
\end{eqnarray*}
which implies the assertion.\\
(iii) For all $x\in N_P$, by Lemma \ref{Lem2} there exist $u,v\in \ck$ with $u-v\in N_P$ such that
\[x=\a(x)(u-v),\ \text{where}\ \a(x)\in \br_+ \ \text{with}\ \a(x)\leq \frac{\l}{2}\|x\| .\] Therefore,
\begin{eqnarray*}
\frac{\|T(x)\|}{\|x\|} &=& \frac{\a(x)}{\|x\|}\|T(u)-T(v)\|\\
&\leq & \frac{\l}{2}\|T(u)-T(v)\|.
\end{eqnarray*}

\noindent Hence, by the definition of $\d_p$ and the previous inequality,  we obtain (\ref{dpuv}).\\

\noindent (iv) Suppose that $H$ is a bounded linear operator on $X$ which commutes with $P$. For all $x\in N_P$, we have
 $$PHx=HPx=0,$$ then $Hx\in N_P$. Therefore,
\begin{eqnarray*}
\|THx\| &\leq & \d_P(T)\|Hx\|\\
& \leq & \d_P(T) \|H\|\|x\|,
\end{eqnarray*}
which implies that
\[\frac{\|THx\|}{\|x\|}\leq \d_P(T)\|H\|,\ \forall \ x\in N_P\] and hence we have $\d_P(TH)\leq \d_P(T)\|H\|$.

\noindent (v) if $H$ is a bounded linear operator on $X$ with $PH=0$, then for all $x\in X$, $Hx \in N_P$. Therefore,
\begin{eqnarray*}
\|THx\| &\leq & \d_P(T)\|Hx\|\\
& \leq & \d_P(T) \|H\|\|x\|,
\end{eqnarray*}
which yields
\[\frac{\|THx\|}{\|x\|}\leq \d_P(T)\|H\|,\ \forall \ x\in X.\]

\noindent (vi) As $S\in \Sigma_P(X)$, we have $Sx\in N_P$, for all $x\in N_P$. Then
 \begin{eqnarray*}
\|T(Sx)\| &\leq & \d_P(T)\|Sx\|\\
& \leq & \d_P(T) \d_P(S)\|x\|,
\end{eqnarray*}
which implies
\[\frac{\|TSx\|}{\|x\|}\leq \d_P(T)\d_P(S),\ \forall \ x\in N_P,\]
then we get
\[\d_P(TS)\leq \d_P(T)\d_P(S),\]
and hence the theorem is proved.
\end{proof}

%
%
% Strong ASS
%
%

Now, let us consider the case of strong abstract state spaces. In this setting, by Theorem (\ref{Yo}), $X_+$ is $1$-generating and the norm is additive on $X_+$. Following the arguments of the proof of
 Lemma \ref{Lem2}, one can prove the next result.

\begin{Lemma}\label{simtoLem2} Let $(X,X_+,\ck,f)$ be a strong abstract state space and let $P$ be a Markov projection. Then for every $x,y\in X$ with $x-y\in N_P$ there
exist $u,v\in \ck$ with $u-v\in N_P$ such that
$$
x-y=\frac{\|x-y\|}{2}(u-v).
$$
\end{Lemma}

Consequently, \eqref{dpuv} can be modified as follows:

\begin{Proposition}\label{DDs} Let $(X,X_+,\ck,f)$ be a strong abstract state space, $P$ be a Markov projection on $X$ and let $T\in \Sigma(X)$. Then:
\begin{equation}\label{simtodpuv} \d_P(T)= \frac{1}{2}\sup \{\|Tu-Tv\|;\ u,v\in \ck \ \text{with}\ u-v\in N_P \}.
\end{equation}
\end{Proposition}

Hence, we have the following result.

\begin{Corollary}\label{dpis0}
Let $(X,X_+,\ck,f)$ be a strong abstract state space, $P$ be a Markov projection on $X$ and $T\in \Sigma(X)$. If $\d_P(T)=0$, then $T=TP$.
\end{Corollary}

\begin{proof}
If $\d_P(T)=0$, then by \eqref{simtodpuv} we have $Tu=Tv$, for all $u,v\in \ck$ with $u-v\in N_P$. As $P$ is a Markov projection, we have $Pu-u\in N_P$. Then
\[Tu=TPu,\ \forall u\in \ck.\]
If $x\in X_+$, then
\[Tx=\|x\|T\left (\frac{x}{\|x\|}\right )=\|x\|TP\left (\frac{x}{\|x\|}\right )=TPx.\]
Now, for all $x\in X$,  $x=x_+-x_-$, ($x_+,x_-\in X_+$). Therefore,
\[Tx=TPx_+-TPx_-=TPx,\]
which proves the assertion.
\end{proof}
From now, we consider general abstract state spaces.  The following proposition is crucial in our investigations.

\begin{Proposition}\label{Prob-1}
Let $(X,X_+,\ck,f)$ be an abstract state space, and let $P$ be a projection on $X$.
If $T\in \Sigma_P(X)$ and $\d_P(T^{n_0})<1$ for some $n_0\in \bn$, then $\|T^n(I-P)\|\to 0$.
\end{Proposition}

\begin{proof}
Given such $n_0\in \bn$ and let $\rho =\d_P(T^{n_0})$. Then for a large $n\in \bn$, we write $n=kn_0+r\ (k,r\in \bn\  \text{and}\ r<n_0)$ and by (vi) of Theorem \ref{Dbp-prp}
$$\d_P(T^n)= \d_P(T^{kn_0}T^r) \leq \rho^k \d_P(T^r).$$  Again using (v) of the same theorem, we have
\[\|T^n(I-P)\|\leq \d_P(T^n)\|I-P\| \leq 2 \rho^k \d_P(T^r)\leq 2\rho^{\lfloor \frac{n}{n_0}\rfloor}  \to 0\  (\text{as}\ n\to \infty),\]
which proves the assertion.
\end{proof}

It is clear that if $T\in \Sigma_P(X)$, then  $T\in\Sigma_{I-P}(X)$. Therefore, it would be interesting to know a relation between
$\d_P(T)$ and $\d_{I-P}(T)$. Next result clarifies this question.

\begin{Proposition}
Let $T\in \Sigma_P(X)$. Then at most one of the following statements is valid:
\begin{enumerate}
\item[(i)] there exists $n_0\in \bn$ such that $\d_P(T^{n_0})<1$;
\item[(ii)] there exists $n_0\in \bn$ such that $\d_{I-P}(T^{n_0})<1$.
\end{enumerate}
\end{Proposition}

\begin{proof}
Suppose that there exist $n_0,m_0\in \bn$ such that
\[\d_P(T^{n_0})<1\ \ \ \text{and}\ \ \ d_{I-P}(T^{m_0})<1.\]

\noindent Then by Proposition \ref{Prob-1}
\[\|T^{n}(I-P)\|\to 0.\]
As  $T\in \Sigma_{I-P(X)}$ and using the same argument
 \[\|T^{n}P\|\to 0.\]
Then
\[\|T^{n}\|= \|T^{n}(P+(I-P))\|\leq \|T^{n}(P)| + \|T^n(I-P)\|\to 0,\]
which contradicts the Markovianity of $T$.
\end{proof}

\begin{Corollary} \label{wpergnotcomp}
If $T\in \Sigma_P(X)$ and $\d_P(T^{n_0})<1$ for some $n_0\in\bn$, then $\d_{I-P}(T^{n})=1$, for all $n\in \bn$.
\end{Corollary}

% Quazi Compact

Let us recall that a bounded linear operator $T$ on a Banach space $X$ is called \textit{quasi-compact} if
there exists an $n_0\in \bn$  such that $\|T^{n_0}-K\|<1$, for some compact operator $K$ on $X$. Quasi-compact operators have been extensively studied in \cite{HH,Lq}.

It is natural to ask: whether $T$ would be a quasi-compact in terms of $\d_P$?  Next result sheds some light on this question.

\begin{Theorem}
 Let $T\in \Sigma_P(X)$  and $TP$  be  quasi-compact on $X$. If there exists an $n_0\in \bn$ such that $\d_P(T^{n_0})<1$, then $T$ is quasi-compact.
  \end{Theorem}

\begin{proof}
The quasi-compactness of $TP$ yields the existence of $m_0\in\bn$  and a compact operator $K$ such that
\[
\|(TP)^{m_0}-K\|<1.
\]

On the other hand, the existence of $n_0\in \bn$  with $\d_P(T^{n_0})<1$,  due to Proposition \ref{Prob-1} implies
\begin{equation}\label{KK1}
\|T^n(I-P)\|=\|T^n-T^nP\|\to 0.
\end{equation}
Then, for any positive $\i$ with $0<\i<1-\|(TP)^{m_0}-K\|$, by \eqref{KK1} one finds $n_1\in\bn$ (we may assume that $n_1>m_0$) such that
$$
\|T^{n_1}-T^{n_1}P\|<\i.
$$
Let $K_1=T^{n_1-m_0}K$, which is clearly compact.
Then
\begin{eqnarray*}
\|T^{n_1}-K_1\|&\leq &\|T^{n_1}-T^{n_1}P\|+\|T^{n_1}P-K_1\|\\[2mm]
&<&\i+\|T^{n_1-m_0}(T^{m_0}-K)\|\\[2mm]
&\leq &\i+\|T^{m_0}-K\|<1,
\end{eqnarray*}
which means that $T$ is quasi-compact.
\end{proof}

From this proposition we immediately get the following one.

\begin{Corollary}
 Let $T\in \Sigma_P(X)$  and $P$  be compact on $X$. If there exists an $n_0\in \bn$ such that $\d_P(T^{n_0})<1$, then $T$ is quasi-compact.
  \end{Corollary}

%
%   Complexification

Let $X$ be an abstract state space. Its complexification $\tilde X$ is defined by $\tilde{X}=X+iX$ with a reasonable norm $\|\cdot\|_\bc$ (see \cite{MST} for details). In this setting,
$X$ is called the \emph{real part} of
$\tilde X$. The \emph{positive cone} of $\tilde X$ is defined as $X_+$.  A vector $f \in \tilde X$ is called \emph{positive}, which
we denote by $f \geq 0$, if $f \in X_+$.  For two elements $f,g\in \tilde X$ we write, as usual, $f \leq g$ if $g-f \geq 0$. In the dual space $\tilde X^*$ of $\tilde X$, one can
introduce an order as follows: a functional $\varphi \in \tilde X^*$ fulfils $\varphi \geq 0$ if and only
if $\langle \varphi, x\rangle \geq 0$ for all $x \in X_+$; we denote the
positive cone in $\tilde X^*$ by $\tilde X^*_+:= (\tilde X^*)_+$.
In what follows, we assume that the norm $\|\cdot\|_\bc$ is taken as
\[\|x+iy\|_\infty = \sup_{0\leq t\leq 2\pi} \|x\cos t-y\sin t\|.\]
We note that all other complexification norms on $\tilde{X}$ are equivalent to $\|\cdot\|_\infty$, and moreover, $\|\cdot\|_\infty$ is the smallest one among all reasonable norms.

A linear mapping $T:X\to X$ can be uniquely extended to $\tilde{T}:\tilde{X} \to \tilde{X}$ by $\tilde{T}(x+iy)=Tx+iTy$. The operator $\tilde{T}$ is called the \textit{extension}
of $T$ and it is well-known that $\|\tilde{T}\|=\|T\|$. In what follows, a mapping $\tilde{T}:\tilde{X}\to \tilde{X}$ is called \textit{Markov}
if it is the extension of a Markov operator $T$. Let $\tilde{P}$ be the extension of a projection $P:X\to X$, and define
\[\tilde{\d }_{\tilde{P}}(\tilde{T})=\sup_{x\in N_{\tilde{P}}}\frac{\|\tilde{T}x\|_\infty}{\|x\|_\infty},\]
where $N_{\tilde{P}}=\{x\in \tilde{X};\ \tilde{P}x=0 \}$.

\begin{Lemma}
Let $X$ be a normed space, $T:X\to X$ be an operator and let $\tilde{T}$ be its extension. Then
\[\tilde{\d }_{\tilde{P}}(\tilde{T})= \d_P(T).\]
\end{Lemma}
\begin{proof}
As $\tilde{T}$ is the extension of $T$, $\tilde{\d }_{\tilde{P}}(\tilde{T})\geq  \d_P(T)$. On the other hand, if $\tilde{x}\in N_{\tilde{P}}\ (\tilde{x}=x+iy)$, then $Px=Py=0$, i.e. both $x$ and $y$ belong to $N_P$. Therefore,
\begin{eqnarray*}
\|\tilde{T}(x+iy)\|_\infty &=& \|Tx+iTy\|_\infty\\
&=& \sup_{0\leq t\leq 2\pi}\|T(x)\cos t -T(y)\sin t \|\\
&=& \sup_{0\leq t\leq 2\pi}\|T(x\cos t -y\sin t) \|\\
&\leq & \d_P(T)\sup_{0\leq t\leq 2\pi}\|x\cos t - y \sin t\|\ (\text{by ($v$) of Theorem \ref{Dbp-prp}})\\
&=& \d_P(T)\|x+iy\|_\infty ,
\end{eqnarray*}
hence $\tilde{\d }_{\tilde{P}}(\tilde{T})\leq  \d_P(T)$, which completes the proof.
\end{proof}

Now, let $S\in \Sigma (X)$ and let $P$ be a projection on $X$. Recall that $X=PX\oplus(I-P)X$ and so the dual $X^*=(PX)^*\oplus((I-P)X)^*$. Assume  that $\l$ is an eigenvalue of $S$, in the following we discuss the comparison between $|\l|$ and $\d_P(T)$.

\begin{Theorem}
Let $P$ be a  Markov projection on a complex space $X$ and let $S \in \Sigma_P(X)$. If one of the following conditions is satisfied:
\begin{enumerate}
\item[(i)]  $\l \neq 1$ is an eigenvalue of $S$ in $(I-\tilde{P})\tilde{X}$; or
\item[(ii)]  $\l \neq 1$ is an eigenvalue of $S^*$ in $((I-\tilde{P})\tilde{X})^*$,
\end{enumerate}
then $|\l|\leq \d_P(S)$.
\end{Theorem}

\begin{proof}
If (i) is satisfied and $x\in (1-\tilde{P})\tilde{X}$ is  a corresponding eigenvector to $\l$ with $\|x\|_\infty=1$, then $x\in N_{\tilde{P}}$ and
\[|\l| =\|\l x\|_\infty=\|\tilde{S}x\|_\infty \leq \sup_{x\in N_{\tilde{P}}}\|\tilde{S}x\|_\infty  \leq \tilde{\d}_{\tilde{P}}(\tilde{S})= \d_P(S).\]
Assume that (ii) is satisfied. Notice that for $y\in \tilde{X}^*$, the set $$\{|y(\tilde{x})|;\ \tilde{x}\in N_{\tilde{P}}\ \text{and}\ \|\tilde{x}\|_\infty\leq 1\}$$ is bounded by $\|y\|$.
Let $G: \tilde{X}^*\to \br$ be defined as follows: \\
\[G(y) =\sup\{|y(\tilde{x})|;\ \tilde{x}\in N_{\tilde{P}}\ \text{and}\ \|\tilde{x}\|_\infty\leq 1\},\ y\in \tilde{X}^*.\]
Now, $\tilde{S}^*y\in \tilde{X}^*$ and
\begin{eqnarray*}
G(\tilde{S}^*y)&=& \sup\{|\tilde{S}^*y(\tilde{x})|;\  \tilde{x}\in N_{\tilde{P}}\ \text{and}\ \|\tilde{x}\|_\infty\leq 1\}\\
&=& \sup\{|y(\tilde{S}(\tilde{x}))|;\ \ \tilde{x}\in N_{\tilde{P}}\ \text{and}\ \|\tilde{x}\|_\infty\leq 1\}\\
&=& \sup\left \{\left |\|\tilde{S}(\tilde{x})\|_\infty y\left(\frac{\tilde{S}(\tilde{x})}{\|\tilde{S}(\tilde{x})\|_\infty}\right )\right |;\ \ \tilde{x}\in N_{\tilde{P}}\ \text{and}\ \|\tilde{x}\|_\infty \leq 1\right \}\\
&\leq & \tilde{\d}_{\tilde{P}}(\tilde{S})  \sup\left \{\left |y\left(\frac{\tilde{S}(\tilde{x})}{\|\tilde{S}(\tilde{x})\|_\infty}\right )\right |;\ \ \tilde{x}\in N_P\ \text{and}\ \|\tilde{x}\|_\infty \leq 1\right \}\\
&\leq & \d_P(S)  \sup\left \{|y(\tilde{v})|;\  \tilde{v}\in N_{\tilde{P}}\ \text{and}\ \|\tilde{v}\|_\infty\leq 1\right \}\ (\text{since}\ \tilde{S}(N_{\tilde{P}})\subseteq N_{\tilde{P}}) )\\
&=& \d_P(S)G(y).
\end{eqnarray*}

If $\l$ is an eigenvalue of $\tilde{S}^*$ in $((I-\tilde{P})\tilde{X})^*$, then for a corresponding eigenvector $\tilde{y}\in ((I-\tilde{P})\tilde{X})^*$ we have
\[ |\l| G(\tilde{y}) =G(\l \tilde{y})=  G(\tilde{S}^*\tilde{y})\leq \d_P(S)G(\tilde{y}).\]
As $\tilde{y}$ is a non-zero eigenvector of $\tilde{S}^*$ which belongs to $((I-\tilde{P})\tilde{X})^*$, there exists $x_0\in (1-\tilde{P})\tilde{X}$ (consequently $x_0\in N_{\tilde{P}}$) such that $\tilde{y}(x_0)\neq 0$. Then we get $G(\tilde{y}) \neq 0$
and hence the proof is completed.
\end{proof}

\begin{remark} We notice that there are many works devoted to the spectral properties of Markov operators (see for example, \cite{A,G}). One of them is its spectral gap. Namely, we say that a Markov operator $T$  on $X$ (here $X$ is a complex abstract state space) has
a \textit{spectral gap}, if one has $\|T(I-P)\|<1$, where $P$ is a Markov projection such that $PT=TP=P$. This is clearly equivalent to $\d_P(T)<1$.  When $X$ is taken as a non-commutative $L_p$-spaces, the spectral gap of Markov operator has been
recently studied in \cite{CPR}. In the classical setting, this gap has been extensively investigated by many authors (see for example, \cite{KM}).

We can stress that if $T$ has a spectral gap, then 1 has to be an isolated point of the spectrum. Indeed, choose an arbitrary $\i>0$ with $\i<1-\d_P(T)$. Assume that $\l$ is an element of the spectrum of $T$ such that $|1-\l|<\i$ with
corresponding eigenvector $x$. Then, it is clear that $y=x-Px$ belongs to $N_P$, therefore, one gets
\begin{eqnarray*}
Ty=Tx-TPx=Tx-PTx=\l(x-Px)=\l y
\end{eqnarray*}
hence, $y$ is an eigenvector with eigenvalue of $\l$, and we have
$$
\|Ty\|=|\l|\|y\|>\d_p(T)\|y\|,
$$
which contradicts to $\d_P(T)<1$.

Going further, we just emphasize that if $T$ has a spectral gap, then one has $\|T^n-P\|\to 0$, which is called as a \textit{uniform $P$-ergodicity}. Next sections will be devoted to this notion.
\end{remark}
\section{Uniformly $P$-ergodic Operators}
In this section, we study uniform $P$-ergodicities of Markov operators on abstract state spaces.

\begin{Definition}
Let $P$ be a projection on $X$. A bounded operator $T:X\to X$ is called uniformly $P$-ergodic if $\|T^n-P\| \to 0$, as $n\to \infty$.
\end{Definition}

\noindent Let us prove the following results for uniform $P$-ergodicity.

\begin{Proposition}
Let $P$ and $Q$ be two  projection operators on $X$  with $Q\leq P$ and let $T\in \Sigma_Q(X)$. If $T$ is uniformly $P$-ergodic, then $TQ$ is uniformly $Q$-ergodic.
\end{Proposition}
\begin{proof}
Suppose that $T$ is uniformly $P$-ergodic. Then $T^n\to P$ as $n\to \infty$, therefore we have
$(TQ)^n=QT^n\to QP=Q$, which proves the statement.
\end{proof}

\begin{Proposition}\label{PTisTPisP}
If $T$ is uniformly $P$-ergodic operator on $X$, then $TP=PT=P$, and in addition, if $T\in \Sigma(X)$, then $P\in \Sigma(X)$.
\end{Proposition}
\begin{proof}
Assume that $T$ is uniformly $P$-ergodic. Then
\[T^{n+1}= TT^n\to TP,\] similarly
\[T^{n+1}= T^nT\to PT,\]
so $PT=TP=P$.

As $T\in \Sigma(X)$, $T^n\in \Sigma(X)$, for all $n\in \bn$.  Therefore, for every $x\in \ck$, one has
$f(Px)=\lim\limits_{n\to \infty}f(T^nx)=1$,  hence $P\in \Sigma(X)$.
\end{proof}

Consequently, in the case of strong abstract state spaces, we deduce the following result.
\begin{Corollary}
Let $(X,X_+,\ck,f)$ be a strong abstract state space, $P$ be a projection on $X$ and let $T\in \Sigma(X)$. If $T$ is uniformly $P$-ergodic and $\d_P(T)=0$, then $T=P$.
\end{Corollary}
\begin{proof}
Directly follows by combining the previous proposition and Theorem \ref{dpis0}.
\end{proof}

\begin{Proposition}\label{Prob-2} Let $(X,X_+,\ck,f)$ be an abstract state space (i.e. $\l$-generating).
If $T$ is uniformly $P$-ergodic, then there exists an $n_0\in \bn$ such that $\d_P(T^{n_0})<1$.
\end{Proposition}

\begin{proof}
The uniformly $P$-ergodicity of $T$ implies the existence of an $n_0\in \bn$ such that
\[\|T^{n_0}-P\|<\frac{1}{2\l}.\]
By (iii) of Theorem \ref{Dbp-prp}, we have
\begin{eqnarray*}
\d_P(T^{n_0}) &\leq& \frac{\l}{2}\sup \|T^{n_0}u-T^{n_0}v\| \ \ \ (\ u,v\in \ck,\ \text{and}\ Pu=Pv)\\
&=& \frac{\l}{2}\sup \|T^{n_0}u-Pu+Pv - T^{n_0}v\|\\
&\leq & \frac{\l}{2}( \sup \|T^{n_0}u-Pu\|+ \sup \|T^{n_0}v-Pv \|)\\
&\leq & \frac{\l}{2}( \|T^{n_0}-P\|+  \|T^{n_0}-P \|)\\
&< &  1,
\end{eqnarray*}
which is the desired assertion.
\end{proof}

Conversely, we have the following theorem:

\begin{Theorem}\label{Thr-2}
Let $T\in \Sigma_P(X)$ be such that $TP=P$. If there exists an $n_0\in \bn$ such that $\d_P(T^{n_0})<1$, then $T$ is uniformly $P$-ergodic.
\end{Theorem}
\begin{proof}
Assume that there exists an $n_0\in \bn$ such that $\d_P(T^{n_0})<1$. By Proposition \ref{Prob-1}
\[\|T^{n}(I-P)\|\to 0,\ \text{as}\ n\to \infty.\]
Therefore,
\[\|T^n-P\|=\|T^n-T^nP\|= \|T^n(I-P)\| \to 0,\]
hence $T$ is uniformly $P$-ergodic.

\end{proof}

\begin{Corollary}\label{unif-p-charac}
Let $T\in \Sigma_P(X)$. Then $T$ is uniformly $P$-ergodic if and only if
\[TP=P\ \text{and} \ \exists n_0\in \bn\ \text{such that} \ \d_P(T^{n_0})<1.\]
  Moreover,
there are constants $C, \a \in \br_+$ and $n_0 \in\bn$ 
such that
$$
\left\| T^n - P \right\| \leq C e^{-\a n}, \,\,\,\,\,
\forall n\geq n_0.
$$
\end{Corollary}

Now, we would like to provide an application of the deduced results above to the case of linear operators which are defined on arbitrary Banach spaces.

\begin{Theorem}\label{TTU}
Let $X$ be any Banach space over $\br$. Assume that $T : X\to X$ is a linear bounded operator with $\|T\|\leq 1$ and $P : X\to X$ is a projection operator with $T P = P T =P$.
Then the following statements are equivalent:
\begin{enumerate}
 \item[(i)] $T$ is uniformly $P$-ergodic;\\
 \item[(ii)]  there is an $n_0\in\bn$ such that $\| T^{n_0}_{|_{I-P}}\|<1$,
 where $T_{|_{I-P}}$ denotes the restriction of $T$ to the subspace $(I-P)(X)$.
 \end{enumerate}
\end{Theorem}

\begin{proof} The implication (i)$\Rightarrow$(ii) is obvious.
Let us prove (ii)$\Rightarrow$(i).
First consider the abstract state space $(\mathcal{X},\mathcal{X}_+,\ck,f)$ which was introduced in Example \ref{E1}-c.
Define the operators $\mathcal{T},\mathcal{P}: \mathcal{X} \to \mathcal{X}$, respectively by
 $$\mathcal{T}(\a , x)=(\a , Tx), \ \  \mathcal{P}(\a , x)=(\a ,  Px).
 $$
 It is clear that $\mathcal{T}$ and $\mathcal{P}$ are Markov operators.
  To prove that $\mathcal{T}$ is uniformly $\mathcal{P}$-ergodic, first we notice that
\[N_{\mathcal{P}}= \{(\a , x)\in \mathcal{X}: \ \mathcal{P}(\a , x)=0 \}=\{(0,x):\ x\in \text{ker}(P) \}.\]
Therefore,
\begin{eqnarray*}
\d_{\mathcal{P}}(\mathcal{T}) &=& \sup \{\|\mathcal{T}(\a , x)\|;\  \|(\a , x)\|\leq 1 \ \text{and}\ (\a , x) \in N_{\mathcal{P}} \}\\
&=& \sup \{\|(0 , Tx)\|;\  \| x\|\leq 1 \ \text{and}\ x \in \text{ker}(P) \}\\
&=& \sup \{\| Tx\|;\  \| x\|\leq 1 \ \text{and}\ x \in (1-P)X\}\\
&=& \|T_{|_{I-P}}\|.
\end{eqnarray*}
Hence, from the condition we infer that $\d_{\mathcal{P}}(\mathcal{T}^{n_0})<1$, then Theorem \ref{Thr-2} implies  $\mathcal{T}$ is uniformly $\mathcal{P}$-ergodic. Using the definition of the norm on $\mathcal{X}$, we obtain the required
assertion.
\end{proof}

 \begin{remark}
A similar kind of result has been proved in \cite{Jach}. An advantage of our approach is that we are working only with $\d_P$, which will allow us to establish some category results for uniformly $P$-ergodic operators
 (see Theorem \ref{unifdensity}).
\end{remark}

We now define a weaker condition than uniform $P$-ergodicity. Namely,  a bounded linear operator $T:X\to X$ is called \textit{weakly $P$-ergodic} if
\[\d_P(T^n)\to 0,\ \text{as}\ n\to \infty .\]

The following result characterizes the concept of weak $P$-ergodicity of $T$.

\begin{Proposition}
Let $T\in\Sigma_P(X)$. Then the following conditions are equivalent:
\begin{enumerate}
 \item[(i)] $T$ is weakly $P$-ergodic;
  \item[(ii)] there exists an $n_0\in \bn$ such that $\d_P(T^{n_0})<1$.
\end{enumerate}
\end{Proposition}

\begin{proof} (i) $\Rightarrow$ (ii) If $T$ is weakly $P$-ergodic, then it is obvious that there exists $n_0\in \bn$ such that $\d_P(T^{n_0})<1$.

(ii) $\Rightarrow$ (i) Assume that such an $n_0\in \bn$ exists and let $\rho =\d_P(T^{n_0})$. Then for a large $n\in \bn$, we write $n=kn_0+r\ (k,r\in \bn\  \text{and}\ r<n_0)$ and by (vi) of
Theorem \ref{Dbp-prp}, we have
$$\d_P(T^n)= \d_P(T^{kn_0}T^r) \leq \rho^k \d_P(T^r).$$
As $n$ tends to 0, $k$ also tends to 0, and hence the proof is completed.
\end{proof}

Using Corollary \ref{wpergnotcomp}, we immediately get the following fact.

\begin{Proposition}
Let $T\in \Sigma_P(X)$. If $T$ is weakly $P$-ergodic, then $T$ is not weakly $(1-P)$-ergodic.
\end{Proposition}

Let us now fix the following notations:
\begin{eqnarray*}
 && \Sigma_P^u(X)=\{T\in \Sigma_P(X):\  T \ \text{is uniformly $P$-ergodic}\}, \\
&&\Sigma_P^w(X)=\{T\in \Sigma_P(X):\  T \  \text{is weakly $P$-ergodic}\},\\
&&\Sigma_P^{inv}(X)=\{T\in \Sigma_P(X):\ TP=P\}.
\end{eqnarray*}
Then, it is clear  that
\begin{eqnarray*}
&&\Sigma_P^u(X)\subseteq \Sigma_P^w(X), \ \ \
\Sigma_P^u(X)\subseteq \Sigma_P^{inv}(X)
\end{eqnarray*}
Moreover,
$$
 \Sigma_P^u(X)= \Sigma_P^w(X)\cap \Sigma_P^{inv}(X).
$$

\begin{Theorem}\label{unifdensity}
 Let $(X,X_+,\ck,f)$ be an abstract state space and let $P$ be a Markov projection on $X$. Then the set $\Sigma_P^{u}(X)$ is
a norm dense and open subset of $\Sigma_P^{inv}(X)$.
\end{Theorem}
\begin{proof}
Given any $T\in\Sigma_P^{inv}(X)$, $0<\i<2$, and let us denote
$$
T^{(\i)}=\bigg(1-\frac{\i}{2}\bigg)T+\frac{\i}{2}P.
$$
It is clear that $T^{(\i)}\in\Sigma_P^{inv}(X)$ and
$$\|T-T^{(\i)}\| =\bigg\|\frac{\i}{2}P-\frac{\i}{2}T\bigg\| = \frac{\i}{2}\|P-T\|<\i. $$
 Now we show that $T^{(\i)}\in\Sigma_P^{u}(X)$. For all $x\in N_P$ by Lemma \ref{Lem2},
 $x=\a(x)(u-v),\ u,v\in \ck$ with $u-v\in N_P$, and $0<\a(x)\leq \frac{\l}{2}\|x\|$. Therefore,
\begin{eqnarray*}
\|T^{(\i)}(x)\|&=&\a(x)\|T^{(\i)}(u-v)\| \\[2mm]
&=&\a(x)\bigg\|\bigg(1-\frac{\i}{2}\bigg)T(u-v)+
\frac{\i}{2}P(u-v)\bigg\|\\[2mm]
&=&\a(x)\bigg(1-\frac{\i}{2}\bigg)\bigg\|T(u-v)\bigg\|\\[2mm]
&=& \bigg(1-\frac{\i}{2}\bigg)\|Tx\|\\
&\leq&\bigg(1-\frac{\i}{2}\bigg)\|x\|,
\end{eqnarray*}
which implies $\d_P(T^{(\i)})\leq 1-\frac{\i}{2}$.
Hence, by Theorem \ref{Thr-2}
$T^{(\i)}\in\Sigma_P^u(X)$.

Now let us show that $\Sigma_P^{u}(X)$ is a norm open subset of $\Sigma_P^{inv}(X)$. First we
establish that for every $n\in\N$, the set
$$
\Sigma_{P,n}^{inv}(X)=\bigg\{T\in \Sigma_P^{inv}(X):  \  \d_P(T^n)< 1\bigg\}
$$
is an open subset of $\Sigma_P^{inv}(X)$. Indeed, take any $T\in\Sigma_{P,n}^{inv}(X)$ and letting
$\a:=\d_P(T^n)<1$, we choose $\b$ such that $0<\b<1$ and $\a+\b<1$. Then, for any
$H\in\Sigma_P^{inv}(X)$ with $\|H-T\|<\b/n$ and using (ii) of Theorem \ref{Dbp-prp}, we
obtain
\begin{eqnarray*}
|\d_P(H^n)-\d_P(T^n)|&\leq&\|H^n-T^n\|\\[2mm]
&\leq& \|H^{n-1}(H-T)\|+\|(H^{n-1}-T^{n-1})T\|\\[2mm]
&\leq& \|H-T\|+\|H^{n-1}-T^{n-1}\|\\[2mm]
&\vdots&\\
&\leq& n\|H-T\|<\b.
\end{eqnarray*}
Hence, the above inequality yields that $\d_P(H^n)<\a+\b<1$, i.e.
$H\in\Sigma_{P,n}^{inv}(X)$.
As
$$
\Sigma_P^{u}(X)=\bigcup_{n\in\N}\Sigma_{P,n}^{inv}(X),
$$
we find that $\Sigma_P^{u}(X)$ is an open subset of $\Sigma_P^{inv}(X)$, which completes the proof.
\end{proof}

\noindent Using the same arguments, one can prove the following theorem.

\begin{Theorem}\label{weakdensity}
 Let $(X,X_+,\ck,f)$ be an abstract state space and let $P$ be a Markov projection on $X$. Then the set $\Sigma_P^{w}(X)$ is
a norm dense and open subset of $\Sigma_P(X)$.
\end{Theorem}

\begin{remark} We notice that the Baire category theorem has a long history in ergodic theory \cite{Hal}, and it has many applications \cite{B0,Iw}.  Baire type
considerations usually bring easy answers to existence problems. In \cite{BK} a particular case of Theorem \ref{unifdensity} has been established for Markov operators, acting on the Schatten
class $C_1$. We aim that our results in this direction will open new perspectives in the non-commutative ergodic theory.
\end{remark}

\section{Characterizations of Uniformly $P$-ergodic Markov operators}

In this section, we provide a large class of examples of uniformly $P$-ergodic operators on abstract state spaces. Precisely, we describe those uniformly $P$-ergodic operators in terms of the projection $P$. Afterwards, we use this characterization to deduce examples of uniformly $P$-ergodic on $\br^n$, on $\ell_1 $ and on $L_1$- spaces.

Let $X$ be an abstract state space. For an operator $Q$ on $X$, let $Rang(Q)$ and $Fix(Q)$ denote the range and the fixed points of $Q$, respectively. We now prove the following auxiliary fact.

\begin{Lemma}
Let $X$ be a vector space, $P$ be a projection operator on $X$ and let $Q$ be any operator on $X$. Then the following statements are equivalent:
\begin{enumerate}
\item[(i)] $Rang(Q)\cap Fix(P)=\{0\}$ and $PQ=QP$;
\item[(ii)] $PQ=QP=0$.
\end{enumerate}
\end{Lemma}

\begin{proof}
$(i) \Rightarrow (ii)$ For every $x\in X$, $QPx\in Rang(Q)$. As
\[P(QPx)= QP^2x=QPx,\]
we get $QPx \in Fix(P)$, then by the assumption  $QPx=0$, and hence assertion (ii) follows.

\noindent $(ii) \Rightarrow (i)$ Suppose that $PQ=QP=0$.  If $x\in Rang(Q)\cap Fix(P)$, then, for some $s\in X$, one has
\[x=Qs\ \ \text{and}\ \ Px=x. \] Therefore,
\[x=Px=P(Qs)=0,\]
which means the assertion (i).
\end{proof}

Now, let us prove the following characterization result.

\begin{Theorem}\label{new-charact-unif-p-erg}
Let $P$ be a projection on $X$. Then $T$ is uniformly $P$-ergodic if and only if $T$ can be written as $T=P+Q$, where $Q$ is an operator on $X$ such that $PQ=QP=0$ and $\|Q^{n_0}\|<1$, for some $n_0\in \bn$. Moreover, if $T\in \Sigma(X)$, then
 \[\d_P(T)\leq \|Q\|\leq 2\d_P(T).\]
 \end{Theorem}

\begin{proof} Suppose that $T$ is uniformly $P$-ergodic. Put $Q=T-P$,  then Proposition \ref{PTisTPisP} implies $PQ=QP=0$. Therefore, $T^n=P+Q^n$. Hence, the uniform $P$-ergodicity implies
the existence
of $n_0\in \bn$ such that
\[\|Q^{n_0}\| =\|T^{n_0}-P\|<1. \]
Conversely, suppose that $T=P+Q$ and $Q$ satisfies the given hypotheses. Then
for every $n\in \bn$, we have
\[T^n=P+Q^n.\]
Therefore,
\[\|T^n-P\|=\|Q^n\| \leq \|Q^{n_0}\|^{[n/n_0]}\to 0\ \text{as}\ n\to \infty ,\]
so $T$ is uniformly $P$-ergodic.

Now assume that $T$ is a Markov operator.  Then
\[ \d_P(T)= \sup_{x\in N_P,\ x\neq
0}\frac{\|Px+Qx\|}{\|x\|}= \sup_{x\in N_P,\ x\neq
0}\frac{\|Qx\|}{\|x\|} =\d_P(Q)\leq \|Q\|.\]

Also, as $T\in \Sigma(X)$ we get $P\in\Sigma(X)$, Therefore, by Proposition \ref{PTisTPisP}
\begin{eqnarray*}
\|Q\| &=& \|T-P\|\\
&=& \|T-TP\| \\
&=& \|T(I-P)\|\\
&\leq & \d_P(T)\|I-P\| \ \ \  (\text{using (v) of Theorem \ref{Dbp-prp})}\\
&\leq & 2\d_P(T),
\end{eqnarray*}
This completes the proof.
\end{proof}

From this theorem, we immediately get  the following result.

\begin{Corollary}\label{sumPandQunif-P-ergod}
Let $X$ be a normed space and let $P$ be a projection on $X$. If $Q$ is an operator on $X$ such that $PQ=QP=0$, then $T=P+\frac{r}{\|Q\|}Q$ is uniformly $P$-ergodic, for all $r\in (-1,1)$ .
\end{Corollary}

%\begin{proof}
%Let $Q$ be an operator on $X$ with $PQ=QP=0$ and let $r\in (-1,1)$. Then
%\[ P\left(\frac{r}{\|Q\|}Q\right ) = \left(\frac{r}{\|Q\|}Q\right )  P= 0,\]
%and $\left \| \left(\frac{r}{\|Q\|}Q\right )\right \|<1$. Hence by Theorem \ref{new-charact-unif-p-erg}, the operator $T$ is uniformly $P$-ergodic.
%\end{proof}

The deduced results above enable us to produce several examples of uniformly $P$-ergodic operators.

\begin{example}
Let us consider $\br^n$ and we denote by $E_i\ (1\leq i\leq n)$ the diagonal matrix units in $\bm_n(\br)$. Then the operator
\[T= \sum_{i=1}^mE_i + \sum_{k=m+1}^{n}r_kE_k,\ \ r_k\in \br \ \text{and}\ |r_k|<1,\]
is uniformly $P$-ergodic, where $P= \sum_{i=1}^mE_i $. As in Theorem \ref{new-charact-unif-p-erg}, we have $Q=\sum_{k=m+1}^{n}r_kE_k$. Indeed, $PQ=QP=0$ and $\|Q\|<1$.
\end{example}

Next example shows that the commutativity of $P$ and $Q$ in Theorem \ref{new-charact-unif-p-erg} is a necessary condition:.

\begin{example}
Let us consider the following operators
\[Q= \begin{pmatrix}
0 &  0  & 0 \\
0 & 0  & 0 \\
 0 &\frac{1}{2} & \frac{1}{4}
\end{pmatrix} \ \text{and} \
P=\begin{pmatrix}
1 &  0  & 0 \\
0 & 1  & 0 \\
 0 & 0 & 0
\end{pmatrix}
.\]
Then $P$ is a projection, $\|Q\|<1$, $PQ=0$ but $QP\neq 0$.  Letting $T=P+Q$, we get that
\[T^n=\begin{pmatrix}
1 &  0  & 0 \\
0 & 1  & 0 \\
 0 &\frac{4^{n+1}-1}{6\cdot4^n} & \frac{1}{4^n}
\end{pmatrix}
\]
 converges to
$$
\tilde P=\begin{pmatrix}
1 &  0  & 0 \\
0 & 1  & 0 \\
 0 & \frac{2}{3} & 0
\end{pmatrix}
.$$
Hence, $T$ is uniformly $\tilde P$-ergodic, but not uniformly $P$-ergodic. Indeed, $T=\tilde P+\tilde Q$, where
$$
\tilde Q= \begin{pmatrix}
0 &  0  & 0 \\
0 & 0  & 0 \\
 0 &-\frac{1}{6} & \frac{1}{4}
\end{pmatrix}
.$$
\end{example}

Next example shows that uniform $P$-ergodicity does not imply quasi-compactness.

\begin{example}
Consider the space $\ell_1$, the subspaces $\ca =\{x\in \ell_1;\ x_{2n}=0\}$ and the operator $P:\ell_1 \to \ca$  defined by
\[P(x)= (x_1+x_2, 0, x_3+x_4, 0, \ldots).\]
Then $P$ is a projection on $\ca$. We construct a class of uniformly $P$-ergodic operators on $\ell_1$ as follows:

Let $\cq:\ell_1 \to \ell_1$ be the operator defined by
\[x \mapsto \left (\frac{-x_2}{2},\frac{x_2}{2}, \frac{-x_4}{2}, \frac{x_4}{2}, \ldots \right ).\]
It is clear that $\cq^{n} \to 0$, so for some $n_0\in \bn$, we have $\|\cq^{n_0}\|<1$. Also, $P\cq=\cq P=0$. Then by Theorem \ref{new-charact-unif-p-erg}, we have that the operator $T=P+\cq$ is uniformly $P$-ergodic,
but one can see that $T$ is not quasi-compact \cite{HH}.
\end{example}

 Now in the following example we construct uniformly $P$-ergodic operators on $L_1$-space:

\begin{example}
Let $(S, \mathcal{B}, \mu)$ be a probability measure space and consider the space $X=L^1(S, \mathcal{B}, \mu)$. We construct a class of uniformly $P$-ergodic operators on $X$ as follows:

Let $f_i(t)\in L^\infty (\mu)$, for $1\leq i\leq n$, and let $E_1$ denote the subspace generated by $span\{f_i\}$. If $P$ is a projection operator from $X$ onto $E_1$, then the operator $P$ can be written as follows
\[(Pf)(t):=\sum_{i=1}^n \Gamma_i(f)f_i(t),\]
where $\Gamma_i$ are linear functionals on $X$, which can be represented as
\[\Gamma_i(f)= \int_S f(t)\gamma_i(t)d\mu ,\  \forall  f\in X\] with
\[\gamma_i\in L^\infty(\mu),\ \text{such that}\  \int_S \gamma_i(t)f_j(t) d\mu =\d_{i,j}.\]

Similarly, let us construct another projection $Q$ on $X$: Let $g_i(t)\in L^\infty (\mu)$, for $1\leq i\leq m$, and let $E_2$ denote the subspace generated by $span\{g_i\}$. Let $Q$ be a projection operator from $X$ onto $E_2$ which is defined by
\[(Qf)(t):=\sum_{i=1}^m \Lambda_i(f)g_i(t),\]
where $\Lambda_i$ are linear functionals on $X$, which can be represented as
\[\Lambda_i(f)= \int_S f(t)\l_i(t)d\mu ,\  \forall  f\in X\] with
\[\l_i\in L^\infty(\mu),\ \text{such that}\  \int_S \l_i(t)g_j(t) d\mu =\d_{i,j}.\]

\noindent In addition, we assume that the choice of $\l_j(t)$ and $\gamma_i(t)$ satisfying
\begin{equation}\label{condition}
 \l_j(t)f_i(t)=0\ \mu\ \text{a.e. and} \ \g_j(t)g_i(t)=0 \ \mu \ \text{a.e.}
\end{equation}

Then $P$ and $Q$ are projections from $X$ onto $E_1$ and $E_2$, respectively. To show that $QP=0$, let $f\in X$ then we have
\begin{eqnarray*}
QP(f) &=& Q\left (\sum_{i=1}^n \Gamma_i(f)f_i(t)\right )\\
&=& \sum_{i=1}^n \Gamma_i(f)Q(f_i(t))\\
&=& \sum_{i=1}^n \Gamma_i(f)\sum_{j=1}^m \Lambda_j(f_i)g_j(t)\\
&=& \sum_{i=1}^n\sum_{j=1}^m \Gamma_i(f) \Lambda_j(f_i)g_j(t)\\
&=& 0,
\end{eqnarray*}
since $\Lambda_j(f_i) = 0$ (see, \eqref{condition}). Similarly, by the second part of \eqref{condition} we get $PQ(f)=0$, for all $f\in X$.
 Therefore,  Corollary \ref{sumPandQunif-P-ergod} implies  that $T=P+rQ$ is a uniformly $P$-ergodic operator on $X$, for all $r\in (-1,1)$ .
\end{example}

\section{On uniform and weak Mean Ergodicities}

In this section, we are going to investigate uniform mean ergodicities of Markov operator.

Given a bounded linear operator $T:X\to X$, we set
$$
A_n(T) =\frac{1}{n}\sum_{k=1}^{n}T^k.
$$
Recall that $T:X\to X$ is said to be
\begin{enumerate}
\item[(a)] \textit{mean ergodic} if for every $x\in X$
\[\lim_{n\to\infty}A_n(T)x=Qx;\]
\item[(b)] \textit{uniformly mean ergodic} if
\[\lim_{n\to\infty}\|A_n(T)-Q\|=0;\]
\end{enumerate}
for some operator $Q$ on $X$.

In this setting, it is well-known that $Q$ is a projection \cite{K}, which is called the \textit{limiting projection of $T$}, and denoted by $Q_T$. Moreover, if $T\in \Sigma (X)$, then $Q_T$ is also Markov.

By analogy with the weak $P$-ergodicity, one may introduce the following notion. A linear operator $T$ is called \textit{weakly $P$-mean ergodic} if
$$
\lim_{n\to\infty}\d_P(A_n(T))=0.
$$
 It is clear that any uniformly mean ergodic operator is weakly $Q_T$-mean ergodic.

By Theorem \ref{Thr-2}, we obtain the following result.

\begin{Corollary}
Assume that  $T\in \Sigma (X)$ and $T$ is mean ergodic with its limiting projection $Q_T$. If there exists an $n_0\in \bn$ such that $\d_{Q_T}(T^{n_0})<1$, then $T$ is uniformly $Q_T$-ergodic.
\end{Corollary}

\begin{Theorem}
Assume that  $T\in \Sigma (X)$ and $T$ is mean ergodic with its limiting projection $Q_T$.  If $T\in \Sigma_P^w(X)$, for some $P$, then $Q_T\leq P$.
\end{Theorem}

\begin{proof}
Suppose that  $T\in \Sigma_P^w(X)$, so $\d_P(T^{n_0})<1$ for some $n_0\in \bn$. Then by Proposition \ref{Prob-1}, we have
 $$\|T^{n}(I-P)\|\to 0.$$
As $TQ_T=Q_T$, $A_n(T)Q_T=Q_TA_n(T)=Q_T$. Then
\begin{eqnarray*}
\|Q_T(I-P)\| &=& \|Q_TA_{n}(T)(I-P)\|
\\
&\leq &\|A_{n}(T)(I-P)\|\\
&\leq & \frac{1}{n}\sum_{k=1}^n\|T^k(I-P)\|\to 0,
\end{eqnarray*}
so $Q_T(I-P)=0$ which implies $Q_T=Q_TP$.

\noindent On the other hand,
\begin{eqnarray*}
\|(I-P)Q_T\| &=& \|(I-P)A_n(T)Q_T\| \\
&\leq & \|(I-P)A_n(T)\| \|Q_T\| \\
&\leq & \|A_n(T)(I-P)\|  \to 0,
\end{eqnarray*}
so $(I-P)Q_T=0$ which implies $Q_T=Q_TP$, and hence $Q_T\leq P$.
\end{proof}

It is natural to ask: when mean ergodic operator would be uniformly mean ergodic? Next result clarifies this question in terms of $\d_P$.

\begin{Theorem}\label{Mergodp}
Assume that  $T\in \Sigma (X)$ and $T$ is mean ergodic with its limiting projection $Q_T$. Then the following statements are equivalent:
\begin{enumerate}
\item[(i)] $T$ is uniformly mean ergodic;
\item[(ii)] there exists an $n_0\in\bn$ such that $\d_{Q_T}(A_{n_0}(T))<1$. Moreover,
$$
\|A_n(T)-Q_T\|\leq \frac{2(n_0+1)}{1-\d_{Q_T}(A_{n_0}(T))}\cdot\frac{1}{n}.
$$
\end{enumerate}
\end{Theorem}

\begin{proof}
We note that if $T=I$, then $Q_T=I$ and according to the definition $\d_{Q_T}(T)=1$, hence the statement of the theorem follows. Therefore, in what follows it is always assumed $T\neq I$.
The implications (i) $\Rightarrow$ (ii) directly follows using the same arguments as in the proof of Proposition \ref{Prob-2}, replacing $T^n$ by $A_n(T)$ and $P$ by $Q_T$.

(ii) $\Rightarrow$ (i). Assume that $\rho = \d_{Q_T}(A_{n_0}(T))<1$, for some $n_0\in \bn$. Then
\begin{eqnarray*}
A_n(T)(I-T)&=& A_n(T)-A_n(T)T \\
&=& \frac{1}{n}\sum_{k=0}^{n-1}T^k- \frac{1}{n}\sum_{k=0}^{n-1}T^{k+1}\\
&=&\frac{1}{n}(I-T^n) ,
\end{eqnarray*}
so,  $\|A_n(T)(I-T)\|\leq \frac{2}{n}$, and then
\[ \|A_n(T)(I-T^k)\|\leq \frac{2k}{n},\ k\in \bn\]
which implies
\begin{eqnarray*}
\|A_n(T)(I-A_{n_0}(T))\|&=&\bigg\|A_n(T)\bigg(\frac{1}{n_0}\sum_{k=1}^{n_0}(I-T^k)\bigg)\bigg\|\\[2mm]
&\leq &\frac{1}{n_0}\sum_{k=1}^{n_0}\|A_n(T)(I-T^k)\|\\[2mm]
&\leq& \frac{n_0+1}{n}.
\end{eqnarray*}
Therefore,
\begin{equation}\label{mer-1}
\d_{Q_T}(A_n(T)(I-A_{n_0}(T))) \leq \frac{n_0+1}{n}.
\end{equation}

Using Properties (ii) and (vi) of Theorem \ref{Dbp-prp}, we have
\begin{eqnarray*}
\d_{Q_T}(A_n(T)(I-A_{n_0}(T)))  &\geq & \d_{Q_T}(A_n(T)) - \d_{Q_T}(A_n(T)A_{n_0}(T))\\
&\geq &  \d_{Q_T}(A_n(T)) - \d_{Q_T}(A_n(T))\d_{Q_T}(A_{n_0}(T))\\
&=& \d_{Q_T}(A_n(T))(1-\rho).
\end{eqnarray*}
By \eqref{mer-1} and as $\rho <1$, we have
\begin{equation}\label{mer-2}
\d_{Q_T}(A_n(T))\leq \frac{n_0+1}{1-\rho}\cdot\frac{1}{n}.
\end{equation}
Now,
\begin{eqnarray*}
\d_{Q_T}(A_n(T)) &=& \sup_{y\in N_{Q_T}}\frac{\|A_n(T)y\|}{\|y\|}\\
&\geq & \sup_{x\in X} \frac{\|A_n(T)x-A_n(T)Q_Tx\|}{\|x-Q_Tx\|}\ \  \ \ \ (\text{for}\ y=x-Q_Tx)\\
&= & \sup_{x\in X} \frac{\|A_n(T)x-Q_Tx\|}{\|x-Q_Tx\|}\ \ \ (\text{since}\ \ A_n(T)Q_T=Q_T)\\
&\geq & \frac{1}{2}\sup_{x\in X} \frac{\|A_n(T)x-Q_Tx\|}{\|x\|}\\
&=& \frac{1}{2} \|A_n(T)-Q_T\|.
\end{eqnarray*}
Then by \eqref{mer-2}
$$
\|A_n(T)-Q_T\|\leq \frac{2(n_0+1)}{1-\rho}\cdot\frac{1}{n}
$$
which yields the desired assertion.
\end{proof}

Now, we are going to introduce an abstract analogue of the
well-known Doeblin's Condition \cite{Num}.

\begin{Definition} Let $(X,X_+,\ck,f)$ be an abstract state space, whose cone $X_+$  is $\l$-generating, let $P$ be a Markov projection on $X$, and let $T\in \Sigma_P (X)$.
We say that $T$ satisfies \textit{condition $\frak{D}_m$} if there exists a constant
$\t\in(0,1]$ and an integer $n_0\in\bn$ and for every $x,y\in \ck$ with $x-y\in N_P$, there exists $z_{xy}\in\ck$ and $\f_{xy}\in X_{+}$ with
$$
\sup\limits_{xy}\|\f_{xy}\|\leq \eta,
$$
where
\begin{equation}\label{Dm0}
0\leq\eta<\t+\frac{1}{\l}-1,
\end{equation}
such that
\begin{equation}\label{Dm}
A_{n_0}(T)x+\f_{xy}\geq\t z_{xy}, \ \ A_{n_0}(T)y+\f_{xy}\geq\t z_{xy}.
\end{equation}
\end{Definition}

The next result characterize the weakly $P$-mean ergodic Markov operators in terms of the above condition $\frak{D}_m$.\\

\begin{Theorem}\label{Dm1}  Let $(X,X_+,\ck,f)$ be an abstract state space whose cone $X_+$  is $\l$-generating, and let $P$ be a Markov projection on $X$.
Assume that $T\in \Sigma_P (X)$. Then the following conditions are equivalent:
\begin{enumerate}
\item[(i)]  $T$ satisfies condition $\frak{D}_m$;

\item[(ii)] there is an $n_0\in\bn$ such that $\d_P(A_{n_0}(T))<1$;

 \item[(iii)] $T$ is weakly $P$-mean ergodic.
\end{enumerate}
\end{Theorem}

\begin{proof} (i) $\Rightarrow$ (ii).
By condition $\frak{D}_m$, there is a $\t\in (0,1]$, $n_0\in\bn$  and for any two elements
$x,y\in\ck$ with $x-y\in N_P$, there exist
$z_{xy}\in\ck$, $\f_{xy}\in X_{+}$ with
\begin{eqnarray}\label{Dm2}
\sup\limits_{xy}\|\f_{xy}\|\leq \eta
\end{eqnarray}
such that
\begin{equation}\label{Dm3}
A_{n_0}(T)x+\f_{xy}\geq\t z_{xy}, \ \ A_{n_0}(T)y+\f_{xy}\geq\t z_{xy}.
\end{equation}

Using the Markovianity of $T$, and the inequalities
\eqref{Dm3} with \eqref{Dm2},  we obtain
\begin{eqnarray*}
\|A_{n_0}(T)x+\f_{xy}-\t z_{xy}\|&=&f(A_{n_0}(T)x+\f_{xy}-\t z_k)\\[2mm]
&=&1-(\underbrace{\t-f(\f_{xy})}_{c})\\[2mm]
&=&1-c\leq 1-(\t-\eta).
\end{eqnarray*}
By the same argument, one finds
\begin{eqnarray*}
\|A_{n_0}(T)y+\f_{xy}-\t z_{xy}\|=1-c\leq 1-(\t-\eta)
\end{eqnarray*}

\noindent Let us denote
\begin{eqnarray*}
x_1=\frac{1}{1-c}(A_{n_0}(T)x+\f_{xy}-\t z_{xy}),\\[2mm]
y_1=\frac{1}{1-c}(A_{n_0}(T)y+\f_{xy}-\t z_{xy}).
\end{eqnarray*}
It is clear that both $x_1,y_1\in \ck$.

So,
\begin{equation*}
\|A_{n_0}(T)x-A_{n_0}(T)y\|=(1-c)\|x_1-y_1\|\leq
2\bigg(1-(\t-\eta)\bigg).
\end{equation*}

Hence,
\begin{equation}\label{N26}
\frac{\l}{2}\|A_{n_0}(T)x-A_{n_0}(T)y\|\leq \l\bigg(1-(\t-\eta)\bigg).
\end{equation}

By \eqref{Dm0} and (iii) of Theorem \ref{Dbp-prp}, and using \eqref{N26}  we obtain,
$$
\d_P(A_{n_0}(T))\leq \m<1 ,
$$
where $\m=\l(1-\t+\eta)$, hence (ii) follows.

The implication (ii) $\Rightarrow$ (iii) immediately follows from the proof of the implication (ii) $\Rightarrow$ (i) of
Theorem \ref{Mergodp}. Therefore, it is enough to establish (iii) $\Rightarrow$ (i).
 Assume that $T$ is weakly $P$-mean ergodic.  Then
\begin{equation*}
\sup_{x,y\in\ck, x-y\in N_P}\|A_n(T)x-A_n(T)y\|\to 0\ \ \ \textrm{as} \ \
n\to\infty.
\end{equation*}
Therefore, one can find $n_0\in\bn$ such that
\begin{equation}\label{N27}
\|A_{n_0}(T)x-A_{n_0}(T)y\|\leq\frac{1}{4\l^2}, \ \ \textrm{for all} \ \ x,y\in\ck, x-y\in N_P.
\end{equation}

Now pick any $y_0\in\ck$ with $x-y_0\in N_p$ and $y-y_0\in N_P$.
Due to Lemma \ref{Lem2} we decompose
\begin{eqnarray}\label{Dm4}
&&A_{n_0}(T)x-A_{n_0}(T)y_0=(A_{n_0}(T)x-A_{n_0}(T)y_0)_+-(A_{n_0}(T)x-A_{n_0}(T)y_0)_-\\[2mm]
&&A_{n_0}(T)y-A_{n_0}(T)y_0=(A_{n_0}(T)y-A_{n_0}(T)y_0)_+-(A_{n_0}(T)y-A_{n_0}(T)y_0)_-.\nonumber
\end{eqnarray}
Denote
$$
\f_{x}=(A_{n_0}(T)x-A_{n_0}(T)y_0)_-, \ \ \f_{y}=(A_{n_0}(T)y-A_{n_0}(T)y_0)_-
$$
and define
$$
\f_{xy}=\f_{x}+\f_{y}.
$$
It is clear that $\f_{xy}\in X_+$ and from \eqref{N27}  with Lemma \ref{Lem2}, one gets
$$
\sup_{x,y\in\ck, x-y\in N_P}\|\f_{xy}\| \leq\frac{1}{4\l}.
$$
Moreover, by \eqref{Dm4} we obtain
\begin{eqnarray*}
A_{n_0}(T)x+\f_{xy}&\geq & A_{n_0}(T)x+\f_{x}\\[2mm]
&=&A_{n_0}(T)y_0+A_{n_0}(T)x-A_{n_0}(T)y_0+\f_{x}\\[2mm]
&=&A_{n_0}(T)y_0+(A_{n_0}(T)x-A_{n_0}(T)y_0)_+\\
 &\geq& A_{n_0}(T)y_0.
\end{eqnarray*}
Similarly, one gets
$$
A_{n_0}(T)x+\f_{xy}\geq A_{n_0}(T)y_0.
$$
Now, by denoting $\t=1$, $\eta=\frac{1}{4\l}$ and $z_{xy}=A_{n_0}(T)y_0$, we infer that the operator  $T$ satisfies the condition $\frak{D}_m$. This completes the proof.
\end{proof}

\begin{remark} We notice that if in the condition $\frak{D}_m$ one replaces $A_n(T)$ with some power of $T$, then we obtain the Deoblin's condition for $T$ which has been investigated in
\cite{DP, M13,M01,SZ}. We think that such type of result is even a new in the classical, i.e.  $X$ is taken as an $L^1$-space.
\end{remark}

In the next example by means of Theorem \ref{Dm1}, we show that weakly $P$-mean ergodic operator is not necessary to be uniformly mean ergodic.

\begin{example} Recall Example \ref{E2} (3). Namely, $X=C[0,1]$ with the cone
$$
X_+=\big\{x\in X: \ \max_{0\leq t\leq 1}|x(t)-x(1)|\leq 2
x(1)\big\}.
$$
Consider the Markov operator $T:X\to X$ given by
$(Tx)(t)=tx(t).$

Let us establish that $T$ satisfies the
condition $\frak{D}_m$. First, it is noted that
$$
(A_n(T)x)(t)=\frac{1}{n}\frac{t-t^{n+1}}{1-t}x(t).
$$
We assume that $Px=x(1)$. Now take $x,y\in\ck$. Put
$\f_{xy}\equiv 0$, $\t=1$ and $z_{xy}=c$, $c\in(0,1/2)$. Then the
inequalities $A_{n_0}x\geq \t z_{xy}$, $A_{n_0}y\geq \t z_{xy}$
are equivalent to $A_{n_0}x-\t z_{xy},A_{n_0}y-\t z_{xy}\in X_+$,
which is equivalent to
\begin{eqnarray*}
&&\max_{0\leq t\leq 1}|(A_{n_0}x)(t)-(A_{n_0}x)(1)|\leq
2\big((A_{n_0}x)(1)-z_{xy}\big), \\[2mm]
&&\max_{0\leq t\leq 1}|(A_{n_0}y)(t)-(A_{n_0}y)(1)|\leq
2\big((A_{n_0}y)(1)-z_k\big).
\end{eqnarray*}
The last one can be rewritten as follows:
\begin{eqnarray*}
&&\max_{0\leq t\leq 1}\bigg|\frac{1}{n_0}\frac{t-t^{n_0+1}}{1-t}x(t)-x(1)\bigg|\leq 2(x(1)-c),
\\[2mm]
&&
 \max_{0\leq
t\leq 1}\bigg|\frac{1}{n_0}\frac{t-t^{n_0+1}}{1-t}y(t)-y(1)\bigg|\leq 2(y(1)-c).
\end{eqnarray*}

Taking
into account $x,y\in\ck$, from the last ones, we have
\begin{eqnarray}\label{1cc}
&&\max_{0\leq t\leq 1}\bigg|\frac{1}{n_0}\frac{t-t^{n_0+1}}{1-t}x(t)-1)\bigg|\leq 2(1-c),
\\[2mm]
&&\label{2cc}
 \max_{0\leq
t\leq 1}\bigg|\frac{1}{n_0}\frac{t-t^{n_0+1}}{1-t}y(t)-1\bigg|\leq 2(1-c).
\end{eqnarray}
From the last expressions, we infer the existence of $n_0$ such that inequalities
\eqref{1cc} and \eqref{2cc} are satisfied. This,  due to Theorem \ref{Dm1},   yields that $T$ satisfies the
condition $\frak{D}_m$. Hence,  $T$ is weakly $P$-mean ergodic. However, one can see that $T$ is not uniformly means ergodic.
\end{example}

Now, we give an application of Theorem \ref{Mergodp}.

 \begin{Theorem} Let $X$ be a Banach space, $T:X\to X$ be a mean ergodic operator on $X$ with $\|T\|\leq 1$ and let $P$ be a Markov projection on $X$.
 Then the following statements are equivalent:
 \begin{enumerate}
 \item[(i)]  there exists an $n_0\in \bn$ such that $\|A_{n_0}(T)_{|_{I-P}}\|<1$;
 \item[(ii)] $T$ is uniformly mean ergodic.
 \end{enumerate}
 \end{Theorem}

 \begin{proof}  (i)$\Rightarrow $(ii). Now consider the abstract state space $(\mathcal{X}, \mathcal{X}_+,\ck,f)$ and the linear operator
 $\mathcal{T}(\a , x)=(\a , T(x))$. Due to Theorem \ref{TTU}, the operator $\mathcal{T}$ is Markov. Moreover, for every $(\a , x)\in \mathcal{X}$, one has
 \begin{eqnarray*}
 A_n(\mathcal{T})(\a , x) &=& \frac{1}{n}\sum_{k=1}^n \mathcal{T}^k(\a , x)\\
 &=& \frac{1}{n}\sum_{k=1}^n (\a , T^k(x))\\
  &=& (\a , A_n(T)(x)).
 \end{eqnarray*}
 Hence, the mean ergodicity of $T$ implies the convergence of $\{A_n(\mathcal{T})(\a, x)\}$, which shows that $\mathcal{T}$ is mean ergodic with its limiting projection $\mathcal{P}$.
By the proof of the implication (ii)$\Rightarrow$(i) in Theorem \ref{TTU}, we have
 \[\d_{\mathcal{P}}(A_n(\mathcal{T}))= \| A_n(T)_{|_{I-P}}\|,\]
 hence, from the hypothesis of the theorem,  for some $n_0\in \bn$, one has
 \[\d_{\mathcal{P}}(A_{n_0}(\mathcal{T}))<1.\]
 So, Theorem \ref{Mergodp} yields that $\mathcal{T}$ is uniformly mean ergodic, which implies the uniform mean ergodicity of $T$.

 The implication (ii)$\Rightarrow $(i) can be proved in the reverse order.

 \end{proof}

 \begin{remark}
We notice that in \cite{LSS} relations between the uniform mean ergodicity  and uniform convergence of the Abel averages have been studied.
\end{remark}

\section{Perturbation Bounds and Uniform $P$-Ergodicity of Markov Operators}

This section is devoted to perturbation bounds for uniformly $P$-ergodic Markov operators. The case when $P$ is a one-dimensional projection, this type of questions have been studied in
\cite{EM2018,Mit,SW}. For general projections, these kinds of bounds have not been investigated. Therefore, results of this section are new even in the classical case as well.

 Recall that if $T$ is uniformly $P$-ergodic, then by Corollary \ref{unif-p-charac}
there are constants $C, \a \in \br_+, n_0 \in\bn$
such that
$$
\left\| T^n - P \right\| \leq C e^{-\a n}, \,\,\,\,\,
\forall n\geq n_0.$$

In this section, we prove perturbation bounds in terms of $C$ and
$e^{\a }$. Moreover, we also give several bounds in terms of
the Dobrushin's ergodicity coefficient.

\begin{Theorem} \label{per1}
Let $(X, X_+, \ck, f)$ be an abstract state space (i.e. $\l$-generating), $P$ be a projection on $X$ and let
$T,S\in \Sigma_P^{inv}(X)$. If $T\in\Sigma_P^u(X)$, then
\begin{eqnarray} \label{1}
\displaystyle && \left\| T^n x- S^n z \right\| \leq
\begin{cases}
\left\| x - z \right\| + n \left\| T-S\right\|, &\forall n \leq \tilde{n},\\
\l C e^{-\a n} \left\| x- z \right\| + \big( \tilde{n}+ \l C \frac{e^{-\a
\tilde{n}} - e^{-\a n}}{1 - e^{-\a}}\big) \left\| T-S\right\|, &
\forall n \geq \tilde{n}+1
\end{cases} \nonumber
\end{eqnarray}
where $\displaystyle \tilde{n}:= 
\bigg[\frac{\log(1/C)}{\log e^{-\a}}\bigg]$, $C, \a \in
\br_+$, $x,z\in \ck$ and $x-z\in N_P$.
\end{Theorem}

\begin{proof}
For every $n\in \bn$,  by induction, we have
\begin{equation} \label{2}
S^n = T^n + \sum_{i=0}^{n-1} T^{n-i-1} \circ (S-T) \circ S^i.
\end{equation}
Let $x,z\in \ck$ and $\ x-z\in N_P$. Then it follows from \eqref{2} that
\begin{eqnarray*}
T^n x - S^n z &=& T^n x - T^n z - \sum_{i=0}^{n-1} T^{n-i-1} \circ (S-T) \circ S^iz \\
&=& T^n  (x - z) - \sum_{i=0}^{n-1} T^{n-i-1} \circ (S-T) (z_i),
\end{eqnarray*}
where $z_i=S^iz$. Hence,
\begin{equation*}
\left\| T^n x  - S^n z \right\| \leq \left\| T^n  (x - z) \right\|
+ \sum_{i=0}^{n-1} \left\|  T^{n-i-1} \circ (S-T) (z_i) \right\|.
\end{equation*}
As $T,S\in \Sigma_p^{inv}(X)$,  we have $P(S-T)=0$ and due to (v) of Theorem (\ref{Dbp-prp}), one finds
\begin{equation*}
 \left\|  T^{n-i-1}  (S-T) (z_i) \right\| \leq \d_P(T^{n-i-1}) \left\| S-T\right\|,
\end{equation*}
and
$$
\left\| T^n (x - z) \right\| \leq \d_P(T^n) \left\| x -z\right\|.
$$
Hence,
\begin{eqnarray} \label{3}
\left\| T^n x  - S^n z \right\| &\leq & \d_P(T^n) \left\| x - z\right\| + \sum_{i=0}^{n-1} \d_P(T^{n-i-1}) \left\| S-T \right\|  \nonumber\\
&=& \d_P(T^n)  \left\|x - z\right\| + \left\| S-T \right\|
\sum_{i=0}^{n-1} \d_P (T^{i}).
\end{eqnarray}
By
$$
\left\| T^i u- T^i v \right\|\leq \left\| T^i u- Pu\right\|+
 \left\|Pv- T^i v \right \|,
$$
with the fact $Pu=Pv$, and due to (iii) of Theorem (\ref{Dbp-prp}), one gets
\begin{eqnarray*}
\displaystyle \d_P(T^i ) \leq  \frac{\l}{2} \sup_{u,v\in\ck,u-v\in N_P} \left\| T^i
u- T^i v \right\| \leq \l \sup_{u\in \ck} \left\| T^i u- P
u\right\|.
\end{eqnarray*}
Therefore,
\begin{eqnarray}\label{33}
\displaystyle \d_P(T^n) \leq
 \begin{cases}
1, &\forall n \leq \tilde{n},\\
\l C e^{-\a n} , &  \forall n \geq \tilde{n}+1
\end{cases}
\end{eqnarray}
where $\tilde{n} = \bigg[\frac{\log (1/C)}{\log e^{-\a}}\bigg] $.

From \eqref{33}  we obtain
\begin{eqnarray} \label{4}
\displaystyle
\sum_{i=0}^{n-1} \d_P(T^i ) &=& \sum_{i=0}^{\tilde{n}-1} \d_P(T^i) + \sum_{i=\tilde{n}}^{n-1} \d_P(T^i) \nonumber\\
&\leq& \tilde{n} + \sum_{i=\tilde{n}}^{n-1} \l C e^{-\a i} \nonumber \\
&=& \tilde{n} + \l C e^{-\a \tilde{n}} \frac{1-e^{-\a (n-
\tilde{n})}}{1- e^{-\a}}, \, \, \forall n \geq \tilde{n}+1 .
\end{eqnarray}
Hence, we get
the required assertion.
\end{proof}

\begin{Corollary}
Assume that the same hypotheses of Theorem \ref{per1} are satisfied. Then, for all $x \in \ck$
\begin{eqnarray}
\displaystyle && \left\| T^n x- S^n x \right\| \leq
\begin{cases}
 n \left\| T-S\right\|, &\forall n \leq \tilde{n},\\
 \big( \tilde{n}+ \l C \frac{e^{-\a
\tilde{n}} - e^{-\a n}}{1 - e^{-\a}}\big) \left\| T-S\right\|, &
\forall n\geq \tilde{n}+1
\end{cases} \nonumber
\end{eqnarray}
here as before,  $\displaystyle \tilde{n}:=
\bigg[\frac{\log(1/C)}{ \log e^{-\a}}\bigg]$, $C,\a \in
\br_+$.
\end{Corollary}

The following theorem gives an alternative method of obtaining
perturbation bounds in terms of $\d_p(T^m)$.
\begin{Theorem} \label{per4}
Let $(X, X_+, \ck, f)$ be an abstract state space, $P$ be a projection on $X$ and let $S,T\in\Sigma^{inv}_P(X)$.
If $\d_P (T^m) < 1$ holds for some $m\in\bn$, then for every
$x,z\in\ck$ with $x-z\in N_P$ one has

\begin{eqnarray} \label{26}
\left\| T^n x - S^n z \right\| &\leq & \d_P(T^m)^{\lfloor n/m
\rfloor} \big(\left\| x - z\right\| + \max_{0< i< m} \left\| T^i -
S^i \right\|\big) \\[2mm]
&&+ \frac{1 - \d_P(T^m)^{\lfloor n/m\rfloor}}{1 - \d_P(T^m)} \left\|
T^m - S^m\right\|, \,\,\,\, n \in \bn \nonumber.
\end{eqnarray}
\end{Theorem}

\begin{proof} For any $n\leq m$, due to $T^n x - S^n z = S^n (x - z) + (T^n - S^n)
x$, we get
\begin{eqnarray} \label{11}
\left\| T^n x- S^n z \right\| &\leq& \left\| x -z\right\| +
 \left\| T^n - S^n \right\|\nonumber\\[2mm]
 &\leq& \left\| x -z\right\| +\max\limits_{0<i<m}
\left\| T^i -S^i\right\|.
\end{eqnarray}
If $n < m$, then Equation (\ref{26}) reduces to \eqref{11}.
If $n\geq m $, we obtain
\begin{eqnarray*}
T^n x - S^n z &=& T^m (T^{n-m} x) - S^m (S^{n-m} z) \\
&=& T^m (T^{n-m} x - S^{n-m} z) + (T^m - S^m) S^{n-m} z.
\end{eqnarray*}
\noindent Therefore, keeping in mind $S,T\in\Sigma^{inv}_P(X)$ one finds
\begin{eqnarray*}
\left\| T^n x - S^n z \right\| \leq \d_P(T^m)\left\| T^{n-m} x -
S^{n-m} z\right\|  + \left\| T^m - S^m \right\|.
\end{eqnarray*}
Applying this relation to
$$\left\| T^{n-m} x - S^{n-m} z\right\|, \cdots,  \left\| T^{n- m(\lfloor n/m\rfloor - 1)} x - S^{n- m(\lfloor n/m\rfloor - 1)} z \right\|
$$
 and using \eqref{11} to bound $\left\| T^{n- m \lfloor n/m\rfloor}
x- S^{n- m \lfloor n/m\rfloor} z \right\|$, we obtain
\begin{eqnarray*} \label{13}
 \left\| T^n x - S^n z \right\| &\leq& \d_P(T^m)^{\lfloor n/m \rfloor} (\left\| x - z\right\| +  \max_{0< i< m} \left\| T^i - S^i \right\|)
 \nonumber\\[2mm]
 && + \bigg(\d_P(T^m)^{\lfloor n/m \rfloor - 1} + \d_P
(T^m)^{\lfloor n/m \rfloor - 2} + \cdots + 1\bigg) \left\| T^m - S^m \right\| \nonumber, \\
&= & \d_P(T^m)^{\lfloor n/m \rfloor} (\left\| x - z\right\| +
\max_{0< i< m} \left\| T^i - S^i \right\|)\\[2mm]
&& + \frac{ 1- \d_P(T^m)^{\lfloor n/m \rfloor}}{1-\d_P(T^m)} \left\|
T^m - S^m \right\|.
\end{eqnarray*}
The proof is completed.
\end{proof}

Consequently, we get the following corollary which allows to estimate the dynamics of $S$ to its fixed points set.

\begin{Corollary} Assume that the same hypotheses of Theorem \ref{per4} are satisfied. Then, for every $x \in \ck$
\begin{eqnarray*} \label{12}
\left\| S^n x-Px \right\| &\leq & \d_P(T^m)^{\lfloor n/m
\rfloor} \big(\left\| x - Px\right\| + \max_{0< i< m} \left\| T^i -
S^i \right\|\big) \\[2mm]
&&+ \frac{1 - \d_P(T^m)^{\lfloor n/m\rfloor}}{1 - \d_P(T^m)} \left\|
T^m - S^m\right\|, \,\,\,\, n \in \bn \nonumber.
\end{eqnarray*}
\end{Corollary}

\section*{Acknowledgments}
The authors thanks Dr. Ho Hon Leung for his help in checking the text of this paper. The authors would like to thank an anonymous referee whose useful suggestions
allowed us to improve the content of the paper.

\bibliographystyle{amsplain}

\begin{thebibliography}{99}

\bibitem{A} E. A. Alekhno, Some properties of essential spectra of a positive pperator, \textit{Positivity}, {\bf 11}(2007),  375--386.

\bibitem{Alf} E.M. Alfsen, \textit{Compact convex sets and booundary
integrals}, Springer-Verlag, Berlin, 1971.

\bibitem{ALM} L. Arlotti, B. Lods, M. Mokhtar-Kharroubi, On perturbed stochastic semigroups on abstract state spaces,
\textit{Z. Anal. Anwend.} \textbf{30} (2011), 457--495.

\bibitem{B0}  W. Bartoszek, Norm residuality of ergodic operators, \textit{Bull. Pol. Ac. Sci. Math} {\bf  29}(1981), 165--167.

\bibitem{B} W. Bartoszek, Asymptotic properties of iterates
of stochastic operators on (AL) Banach lattices, \textit{Anal.
Polon. Math.} {\bf 52}(1990), 165-173.

\bibitem{BK} W. Bartoszek,  B. Kuna, On residualities in the set of Markov operators on $C_1$,
\textit{Proc. Amer. Math. Soc.} {\bf 133} (2005), 2119--2129.


%\bibitem{CRS} N. Crawford, W. De Roeck, M. Schutz, Uniqueness regime for Markov dynamics on quantum lattice spin systems,
%\textit{J. Phys. A: Math. Theor.} {\bf 48} (2015), no. 42, 425203.

\bibitem{C} J. E. Cohen, Y. Iwasa, G. Rautu, M.B. Ruskai, E.
Seneta, G. Zbaganu,  Relative entropy under mappings by stochastic
matrices, \textit{Linear Algebra Appl.} {\bf 179}(1993), 211-235.

\bibitem{CPR} J. M. Conde-Alonso, J. Parcet, E. Ricard, On spectral gaps of Markov maps, \textit{Israel J. Math.} {\bf 226}(2018), 189--203.

%\bibitem{CLMP} T. S. Cubitt, A. Lucia, S. Michalakis, D. Perez-Garcia, Stability of Local Quantum Dissipative Systems,
%\textit{ Comm.  Math.  Phys. } {\bf 337} (2015), no. 3,  1275-1315.


\bibitem{D} R. L. Dobrushin, Central limit theorem for
nonstationary Markov chains. I,II, \textit{Theor. Probab. Appl.}
{\bf 1}(1956),65--80; 329--383.


\bibitem{DP} C. C. Y. Dorea, A. G .C. Pereira, A note on a variation of Doeblin’s condition
for uniform ergodicity of Markov chains, \textit{Acta Math. Hungar.}, {\bf 10} (2006),
287--292.

\bibitem{E}  E. Yu. Emel'yanov, \textit{ Non-spectral asymptotic analysis of one-parameter operator semigroups},
Birkh\"{a}user Verlag, Basel, 2007.



%\bibitem{EM2018}  N. Erkursun-Ozcan, F. Mukhamedov,
 %Uniform ergodicities of Lotz -
%R\"abiger nets of Markov operators on ordered Banach spaces,
%\textit{Results Math.} {\bf 73} (2018), no. 1, 35.


%\bibitem{EW1} E. Yu. Emel'yanov, M.P.H. Wolff, Positive operators on
%Banach spaces ordered by strongly normal cones,
%\textit{Positivity} {\bf 7}(2003), 3--22.

\bibitem{EW2} E. Yu. Emel'yanov, M.P.H. Wolff, Asymptotic behavior of
Markov semigroups on non-commutative $L_1$-spaces, In book:
\textit{Quantum probability and infinite dimensional analysis}
(Burg, 2001), 77--83, QP--PQ: Quantum Probab. White Noise Anal., 15,
World Sci. Publishing, River Edge, NJ, 2003.

\bibitem{EM2017}  N. Erkursun-Ozcan, F. Mukhamedov, Uniform ergodicities and perturbation bounds of Markov chains on
ordered Banach spaces, \textit{J. Phys.: Conf. Ser.} {\bf
819}(2017), 012015.

\bibitem{EM2018} N.  Erkursun-Ozcan, F. Mukhamedov,  Uniform ergodicities and
perturbation bounds of Markov chains on ordered Banach spaces,
\textit{Queast. Math.} {\bf 41} (2018), no. 6,  863--876. 

%\bibitem{EN} K. J. Engel, R. Nagel, \textit{A short course on operator semigroups}, Springer, New York, 2006.

%\bibitem{FR1} F. Fagnola, R. Rebolledo, On the existance of stationary states for
%quantum dyanamical semigroups, \textit{Jour. Math. Phys.}
%\textbf{42} (2001), 1296--1308.

%\bibitem{F}D. S. Fran\c{c}a, Perfect Sampling for Quantum Gibbs States, arXiv:1703.05800.


\bibitem{GQ} S. Gaubert, Z. Qu, Dobrushin's ergodicity
coefficient for Markov operators on cones and beyond,
\textit{Integ. Eqs. Operator Theor.} {\bf 81}(2014), 127--150.

\bibitem{G} J. Gl\"{u}ck, On the peripheral spectrum of positive operators, \textit{Positivity}, {\bf 20}(2016), 307--336.

\bibitem{Hal} P.R. Halmos, \textit{Lectures on ergodic theory}, Chelsea, New York, 1960.

\bibitem{H}  D.J. Hartfiel. Coefficients of ergodicity for imprimitive marices, \textit{Commun. Statis. Stochastic Models} {\bf 15}(1999), 81--88.

\bibitem{HR}  D.J. Hartfiel, U.G. Rothblum,  Convergence of inhomogeneous products of matrices and coefficients of ergodicity,
\textit{Lin. Alg. Appl.} {\bf 277}(1998), 1--9.

\bibitem{HH} H. Hennion, L. Harve, \textit{Limit theorems for Markov chains and stochastic properties of dynamical systems by
quasi-compactness}, Lec. Notes Math.  {\bf 1766} (2001), Springer-Verlag, Berlin.


\bibitem{IS} I.C.F. Ipsen, T.M. Salee, Ergodicity coefficients
defined by vector norms, \textit{SIAM J. Matrix Anal. Appl.} {\bf
32}(2011), 153--200.

\bibitem{Iw} A. Iwanik, Baire category of mixing for stochastic operators, \textit{Rend. Circ. Mat. Palermo, Serie II}
{\bf 28} (1992), 201--217.

\bibitem{Jach} J. Jachymski, Convergence of iterates of linear operators and the Kelisky-
Rivlin type theorems, \textit{Studia Math.}, {\bf 195} (2009), 99--113.

\bibitem{Jam} G. Jameson, \textit{Ordered linear spaces}, Lect.
Notes  Math. V. 141, Springer-Verlag, Berlin, 1970.

%\bibitem{Kar} N.V. Kartashov, Inequalities in theorems of ergodicity and stability for Markov chains with common Phase space,
%I, \textit{Probab. Theor. Appl.} {\bf 30}(1986), 247--259.


\bibitem{KM} I. Kontoyiannis, S.P. Meyn, Geometric ergodicity and the spectral gap of non-reversible Markov chains,  \textit{Probab. Theory Relat. Fields}
{\bf 154}(2012), 327--339.

\bibitem{K} U. Krengel, \textit{Ergodic Theorems}, Walter de Gruyter,
Berlin-New York, 1985.

%\bibitem{LM} A. Lasota, T.Y. Li, J.A. Yorke, Asymptotic periodicity of the iterates of Markov operator. \textit{Trans. Amer.
%Math. Soc.} {\bf 286}(1984), 751--764.

\bibitem{L} M. Lin,  On the uniform ergodic theorem, \textit{Proc. Amer. Math. Soc.} {\bf 43} (1974),  337--340.

\bibitem{Lq} M. Lin,  On quasi compact Markov operators, \textit{Ann of Probab.} {\bf 2} (1974),  464--475.

\bibitem{LSS} M. Lin, D.  Shoikhet, L. Suciu, Remarks on uniform ergodic theorems, \textit{ Acta Sci. Math. (Szeged)} {\bf 81} (2015),  251--283.

\bibitem{MZ}  M. Mbekhta, J. Zemanek, Sur le theoreme ergodique uniforme et le spectre,  \textit{C.R. Acad. Sci. Pais Ser. I. Math.}  {\bf 317} (1993), 1155--1158.


\bibitem{MT}  S.P. Meyn, R.L. Tweedie, \textit{Markov chains and stochastic stability}, Springer-Verlag, 1996.

\bibitem{Mit} A. Mitrophanov, Sensitivity and convergence of
uniform ergodic Markov chains, \textit{J. Appl. Probab.} {\bf 42}
(2005), 1003--1014.

\bibitem{Mit1} A. Mitrophanov, Stability estimates for finite homogeneous continuous-time Markov chains, \textit{ Theory Probab. Appl.} {\bf 50}
(2006), no. 2, 319--326

\bibitem{M} F. Mukhamedov, Dobrushin ergodicity coefficient and ergodicity of noncommutative
Markov chains, \textit{J. Math. Anal. Appl.} {\bf 408} (2013),
364--373.

\bibitem{M13} F. Mukhamedov, On $L_1$-weak ergodicity of non-homogeneous discrete
Markov processes and its applications, \textit{Rev Mat Complut} {\bf 26} (2013),  99--813.

\bibitem{M0} F. Mukhamedov, Ergodic properties of nonhomogeneous Markov chains defined on ordered Banach spaces with a
base, \textit{Acta. Math. Hungar.} {\bf 147} (2015), 294--323.

\bibitem{M01} F. Mukhamedov, Strong and weak ergodicity of
nonhomogeneous Markov chains defined on ordered Banach spaces with a
base, \textit{Positivity} {\bf 20}(2016), 135--153.


\bibitem{MST}  G.A. Munoz, Y. Sarantopoulos, A. Tonge, Complexification of real Banach spaces, polynomials and miultilinear maps, \textit{Studia Math.} {\bf 134}(1999), 1-33.

\bibitem{Num} E. Nummelin, \textit{General Irreducible Markov Chains and Non-negative Operators}, Cambridge University
Press, Cambridge 1984.


\bibitem{R} F. Rabiger, Stability and ergodicity of dominated semigroups, I. The uniform case, \textit{Math. Z.} {\bf 214}(1993), 43--53.


%\bibitem{RKW} D. Reeb, M. J. Kastoryano, M. M. Wolf, Hilbert's projective
%metric in quantum information theory, \textit{J. Math. Phys.} {\bf
%52} (2011), 082201.

\bibitem{SZ} T.A. Sarymsakov, N.P.
Zimakov, Ergodic principle for Markov semi-groups in ordered
normal spaces with basis, \textit{Dokl. Akad. Nauk. SSSR} {\bf
289} (1986), 554--558.

\bibitem{Se2} E. Seneta, \textit{Non-negative matrices and Markov chains},
Springer, Berlin, 2006.

\bibitem{SW} O. Szehr, M.M. Wolf, Perturbation bounds for quantum
Markov processes and their fixed points, \textit{J. Math. Phys.}
{\bf 54}(2013), 032203.

\bibitem{TZ} Y. Tomilov, J. Zemanek, A new way of constructing examples in operator ergodic theory, \textit{Math. Proc. Cambridge Phil. Soc.} {\bf 137} (2004), 209--225.


\bibitem{Yo} D. Yost, A base norm space whose cone is not 1-
generating, \textit{Glasgow Math. J.} {\bf 25} (1984), 35--36.


\bibitem{WN}  Y. C. Wong, K. F. Ng, \textit{Partially ordered topological vector spaces}, Clarendon Press, 1973.



\end{thebibliography}

\end{document}